\documentclass[11pt,reqno]{amsart}
\numberwithin{equation}{section}

\usepackage{amsmath,amsfonts,amssymb,amscd,amsthm}
\usepackage{latexsym}
\usepackage{graphicx}
\usepackage[usenames,dvipsnames,svgnames,table]{xcolor}
\usepackage{fourier}

\newtheorem{theorem}{Theorem}[section] 

\newtheorem{lemma}[theorem]{Lemma}
\newtheorem{proposition}[theorem]{Proposition}
\newtheorem{remark}[theorem]{Remark}

\newtheorem{corollary}[theorem]{Corollary}
\newtheorem{claim}{Claim}

\newcommand{\f}[2]{\frac{#1}{#2}}
\newcommand{\dpr}[2]{\langle #1,#2 \rangle}

\newcommand{\al}{\alpha}
\newcommand{\be}{\beta}

\newcommand{\de}{\delta}

\newcommand{\De}{\Delta}

\newcommand{\la}{\lambda}

\newcommand{\vp}{\varphi}

\newcommand{\rn}{{\mathbb R}}

\newcommand{\rone}{\mathbb R}
\newcommand{\rtwo}{\mathbb R^2}

\newcommand{\Kf}{K_{\frac{1}{4}}}
\newcommand{\Kt}{K_{\frac{1}{2}}}

\newcommand{\cl}{\mathcal L}

\newcommand{\p}{\partial}

\DeclareMathOperator{\Diff}{D}
\DeclareMathOperator{\F}{{\mathcal F}}
\DeclareMathOperator{\supp}{supp}

\newcommand{\I}{\mathcal I}
\newcommand{\J}{\mathcal J}
\newcommand{\N}{\mathbb N}

\newcommand{\dx}{\, \mathrm{d}x}
\newcommand{\dy}{\, \mathrm{d}y}
\newcommand{\dz}{\, \mathrm{d}z}
\newcommand{\dxi}{\, \mathrm{d}\xi}

\newcommand{\beq}{\begin{equation}}
\newcommand{\eeq}{\end{equation}}
\newcommand{\beqna}{\begin{eqnarray*}}
\newcommand{\eeqna}{\end{eqnarray*}}
\newcommand{\beqn}{\begin{equation*}}
\newcommand{\eeqn}{\end{equation*}}
\newcommand{\bp}{\begin{proof}}
\newcommand{\ep}{\end{proof}}
\newcommand{\bprop}{\begin{proposition}}
\newcommand{\eprop}{\end{proposition}}
\newcommand{\bt}{\begin{theorem}}
\newcommand{\et}{\end{theorem}}
\newcommand{\bex}{\begin{Example}}
\newcommand{\eex}{\end{Example}}
\newcommand{\bc}{\begin{corollary}}
\newcommand{\ec}{\end{corollary}}
\newcommand{\bcl}{\begin{claim}}
\newcommand{\ecl}{\end{claim}}
\newcommand{\bl}{\begin{lemma}}
\newcommand{\el}{\end{lemma}}

\usepackage{color}

\usepackage[colorlinks=true]{hyperref}
\hypersetup{urlcolor=blue, citecolor=red, linkcolor=blue}
\usepackage{cleveref}

\begin{document}

\title[Solitary waves for the Whitham equation]{A maximisation technique for solitary waves: the case of the nonlocally dispersive Whitham equation}

\author{Mathias Nikolai Arnesen}
\author{Mats Ehrnstr\"om}
\author{Atanas G. Stefanov}

\address{Mathias Nikolai Arnesen, Department of Mathematical Sciences, Norwegian University of Science and Technology, 7491 Trondheim, Norway}
\email{mathias.arnesen@ntnu.no}

\address{Mats Ehrnstrom, Department of Mathematical Sciences, Norwegian University of Science and Technology, 7491 Trondheim, Norway}
\email{mats.ehrnstrom@ntnu.no}

\address{Atanas G. Stefanov\\ 
 Department of Mathematics,
 	University of Alabama - Birmingham, 
 	University Hall, Room 4005, 
 	1402 10th Avenue South
 	Birmingham AL 35294-1241, USA. 
 }
 \email{stefanov@uab.edu}

\thanks{ME acknowledges the support by grant nos. 250070 and 325114 from the Research Council of Norway. Part of this research was carried out during the programme \emph{Geometric aspects of nonlinear partial differential equations} at Institut Mittag-Leffler, Stockholm. AGS is supported in part by NSF-DMS \# 1908626 and NSF-DMS \# 2204788.}

   
\date{\today}

\begin{abstract}
Recently, two different proofs for large and intermediate-size solitary waves of the nonlocally dispersive Whitham equation have been presented, using either global bifurcation theory or the limit of waves of large period. We give here a different approach by maximising directly the dispersive part of the energy functional, while keeping the remaining nonlinear terms fixed with an Orlicz-space constraint. This method is to our knowledge new in the setting of water waves. The constructed solutions are bell-shaped in the sense that they are even, one-sided monotone, and attain their maximum at the origin. The method initially considers weaker solutions than in earlier works, and is not limited to small waves: a family of solutions is obtained, along which the dispersive energy is continuous and increasing. In general, our construction admits more than one solution for each energy level, and waves with the same energy level may have different heights. Although a transformation in the construction hinders us from concluding the family with an extreme wave, we give a quantitative proof that the set reaches `large' or 'intermediate-sized' waves.
\end{abstract}

\maketitle

\section{Introduction}
We propose in this paper a new and somewhat different variational method towards solitary waves in dispersive equations. Solitary waves in nonlinear dispersive settings go back to the discoveries of John Scott Russell, Boussinesq, Lord Rayleigh and Korteweg-de~Vries in the 1800s, but a rigorous theory emerged first with works such as Lavrentieff's \cite{MR0048985}, Friedrich--Hyers' \cite{MR65317} and Ter-Krikorov's \cite{MR145776} in the mid 20th century. Early on, limiting procedures for long but small waves, and perturbative expansions around flowing parallel streams were used. The small-amplitude theory was later extended by Beale using the Nash--Moser implicit function theorem \cite{MR445136}, with which he established smooth parameter-dependence for all waves of small amplitude. Amick and Toland similarly extended the existence results from small to large amplitudes, yielding for the water-wave equations a full family of waves with increasing maximal slope all the way up to the wave of greatest height \cite{MR647410, MR629699}. Although variational formulations of the water wave problem existed much earlier, Turner seems to have been the first to develop a variational existence theory, simultaneously for periodic and solitary waves \cite{MR766131}, based on joint work with Bona and Bose \cite{MR735931}. The latter two authors also gave a large-amplitude functional-analytic theory for solitary waves in nonlinear equations together with Benjamin \cite{MR1062564}. At this time, Lions introduced the method of concentration-compactness \cite{MR778970}, and Weinstein made use of it in his constructions of shallow-water solitary waves using constrained minimisation \cite{MR886343}. Trivial solutions of water-wave equations are generally global minimisers, so critical points are found as constrained minimisers or saddle points; a common construction has been \(L^2\) or otherwise quadratically constrained minimisation in Sobolev spaces, see for example Buffoni \cite{Buffoni04a}, Groves--Wahlén \cite{MR2847283} and Buffoni--Groves--Sun--Wahlén \cite{MR2997362} on waves with surface tension, among other works.

For nonlinear equations with negative-order dispersion, such as the Whitham equation, 
\begin{equation}\label{eq:whitham}
\eta_t + K*\eta_x+ \p_x( \eta^2) =0,
\end{equation}
with a convolution kernel $K$ given by the Fourier symbol
\begin{equation}\label{eq:m}
m(\xi)=\left(\f{\tanh(\xi)}{\xi}\right)^{\frac{1}{2}},
\end{equation}
solitary waves were first obtained in \cite{EGW11} by Ehrnström--Groves--Wahlén. That proof uses a combination of \(L^2\)-constrained variation, minimisation sequences built on perodic minimisers, and Lions' concentration–compactness technique. The symbol \eqref{eq:m} describes the linear dispersion of gravity surface waves in water of finite depth, and as made precise by Emerald \cite{MR4321411}, the equation can be shown to be in this sense slightly superior to KdV and other models of the same order in the small-amplitude long-wave regime.  
Different and modified perturbative proofs for small-amplitude solitary waves in \eqref{eq:whitham} have since been suggested, for example by Stefanov--Wright \cite{MR4061635}, who used an implicit-function type theorem, and Hildrum \cite{MR4072387}, who also included more irregular nonlinearities. But it was first with Truong--Wahlén--Wheeler \cite{MR4458407} and Ehrnström--Nik--Walker \cite{MR4531652} that large solitary waves for the Whitham equation were found. The regularity theory for highest waves, including solitary such, had been established by Ehrnström--Wahlén in \cite{MR4002168}, and so had symmetry and decay by Bruell--Ehrnström--Pei in \cite{MR3603270}, but there were no proofs of large-amplitude existence in the solitary case. 

The proof presented in the paper at hand should be viewed as an alternative construction, for the Whitham equation, and for dispersive model equations in general. It is inspired by a method by Stefanov--Kevrekidis on solitary waves in Hertzian and monomer chains from \cite{MR3032795, MR3009116}, which we have  adapted to the weakly dispersive case. Among the things that make the dispersive theory different, is that the linear operator and equation are of opposite character from \cite{MR3032795, MR3009116}, and that the theory for `large' solutions is unique to the dispersive case. 

Returning to our problem, we shall be interested in travelling wave solutions $\eta(t,x)=\vp(x- \mu t)$, so that $\lim_{|z|\to \infty} \vp(z) =0$. This yields the nonlinear nonlocal equation 
\begin{equation}\label{eq:steady_whitham}
-\mu \vp+ K*\vp +\vp^2 =0.
\end{equation}
A function  \(\varphi\) shall be called a \emph{steady Whitham solution} if it satisfies \eqref{eq:steady_whitham} almost everywhere. If \(\varphi\) is furthermore in \(L^p(\rn)\), for some \(p \in [1,\infty)\), we shall call it a \emph{solitary} wave. The idea behind our approach is the following: if the energy functional for a nonlinear equation is of the structural form
\[
\int \varphi^2 + \int \varphi L \varphi + \int N(\varphi),
\] 
with \(L\) a linear dispersive operator and \(N(\varphi)\) a primitive of the nonlinearity in the original equation, then critical points in a subspace of \(L^2\) may be found through more than one form of constraint. Most commonly, the term \(\int \varphi^2\) is fixed while minimising the remaining part of the functional; this yields an energetic type of stability, as described by Mielke in the work \cite{MR1949971}. But the nonlinear term \(\int N(\varphi)\), too, could be taken as a constraint, see for example Arnesen \cite{MR3485840} and Zeng \cite{MR1954506}. \emph{We for our part shall keep the combination \( \int \varphi^2 + \int N(\varphi)\) fixed, while maximising the dispersive part \(\int \varphi L \varphi\).} This is made possible by working in an Orlicz space, which essentially equates the nonlinear function \( \Psi(\varphi) = \varphi^2 + N(\varphi)\) with the function \(|\varphi|^p\) in \(L^p\)-theory. For this to work, however, the above \(\Psi\) will be cut off and extended at a point corresponding to the appearance of highest waves, which also coincides with the point where \(\Psi\) seizes to be convex.

We suggest this method as a possible alternative to other existence methods for solitary dispersive waves. Its distinct features are: it is based on \(L^p\)-theory instead of Sobolev theory; it immediately provides bell-shapedness of constrained maximisers (they are even, positive and one-sided monotone); and it is not restricted to solutions of small sizes. In our particular case, we have not been able to prove that we reach the highest wave via this method, but this seems to be a difficulty rather than a constraint of the method; and we establish that we reach at least medium-sized waves. 

The paper starts in Section~\ref{sec:preliminaries} with a walk-through of properties of the convolution kernel \(K\), the functional, and Orlicz spaces. Some of these results are new, other are by alternative proofs. We introduce the symmetric rearrangement of a function, the Riesz convolution--rearrangement inequality, and the necessary prerequisites to prove that a maximiser of our functional solves the correct Euler--Lagrange equation.

Section~\ref{sec:variational} introduces the functional
\[
\J(f) = \langle f, K\ast f\rangle^\frac{1}{2}
\]
with the constraint that \( \int ( \alpha f^2-\frac{1}{3} f^3) \dx = 1\) for \(0\leq f \leq \alpha\) (and something else when \(f > \alpha\)). For each fixed \(\alpha>0\) we are looking for solutions in \(L^2 \cap L^3\). Large values of \(\alpha\) will correspond to small solutions in the original problem, and very small values of \(\alpha\) will show not to correspond to physical solutions at all. By setting up a sequence of problems for \(\textup{supp}\ f\subset [-2^l, 2^l]\), we show that the maximum \(J_l\) is approximately \(\alpha^{-1/2}\), with some sharper estimates. For each \(l\) large enough, and \(\alpha > \alpha_0 > 0\), we find a bell-shaped maximiser $f_l$ that satisfies the Euler--Lagrange equation
\[
	K* f_l(x)=\f{J_l^2}{\dpr{f_l}{\Psi'(f_l)}} \Psi'(f_l(x)), \quad -2^l<x<2^l,
\]
and fulfils \(f_l(0) < \alpha\) with a non-degenerate condition. This is Lemma~\ref{lemma:local Euler}. The hardest part of the paper is Prop.~\ref{prop:upper bound on f_l}, a detailed rearrangement proof to show that maximisers are bounded from above by a value close to \(\alpha\), a result that later gives us some control on the size of solutions. The proof is technical, though rudimentary in technique, and is based on moving parts of the mass of a possible maximiser with too large supremum to obtain a contradiction. The proof contains a Slobodeckij-type difference characterisation \eqref{eq:slobodeckij} of the quadratic form \(\langle f, K\ast f\rangle\), itself equivalent to a squared \(H^{-1/4}(\rone)\)-norm, that makes it possible to relate `flatness' of \(f\) to the functional \(\J(f)\). Section~\ref{sec:variational} additionally covers the limit \(l \to \infty\), which is achieved through weak convergence arguments and properties of the compactly supported maximisers. As in many other investigations considering long waves converging towards solitary solutions, the limit as \(l \to \infty\) need not be unique.

Section~\ref{sec:alpha} is about the dependence of maximisers upon the parameter \(\alpha\). We prove that \(\alpha \to \infty\) corresponds to the small-amplitude limit of the maximisers, while \(\alpha \to 0\) gives maximisers that are outside the valid regime of the Euler--Lagrange equation. An estimate for the largest obtained waves is given. At the threshold value \(\alpha = \alpha_0\), there is a solution satisfying \(f(0) \geq \alpha_0\), and another satisfying \(g(0) \leq \alpha_0\), but because of lack of uniqueness we cannot exclude that these are different, to reach the desired conclusion that in fact \(f(0) = g(0) = \alpha_0\). That would correspond to a highest wave. 

Finally, we go back to the original equation through the somewhat implicit transformation 
\(
\varphi = \f{J_{\alpha}^2 f_{\alpha}}{2-\f{1}{3}\int f_{\alpha}^3 \dx}.
\)
The main result and estimates can be found in Theorem~\ref{thm:main}. We find an injective curve of bell-shaped solutions parameterised by \(\alpha \in [\alpha_0,\infty)\), with the function \(\alpha \mapsto \alpha J_\alpha^2  \in (1, {\textstyle \frac{3}{2}})\) strictly decreasing and continuous with unit limit as $\al\rightarrow \infty$. We give some \(L^p\)-estimates of these waves, and show that the small waves converge to the expected bifurcation point for solitary waves.

\section{Preliminaries}\label{sec:preliminaries} 
Throughout this paper, \(\lesssim\), \(\gtrsim\) and \(\eqsim\) shall indicate (in)equalities that hold up to uniform positive factors. When the factors involved depend on some additional parameter or function, this will be indicated with subscripts such as \(\gtrsim_\mu\). We shall call an element in \(L^p(\rone)\) \emph{bell-shaped} if it lies in the closure of the set of even, continuous and positive functions which are decreasing on the positive half-axis. This is equivalent to requiring the same properties almost everywhere for a general, not necessarily continuous, function in \(L^p(\rone)\). We furthermore define the Fourier transform \(\mathcal F\) by
$$
\mathcal F(f)(\xi) =\int_{\rone} f(x) e^{- i x\xi} \dx,	\qquad f \in S(\rone),
$$
extended by duality from the Schwartz space \(S(\rone)\) of rapidly decaying smooth functions to \(S^\prime(\rone)\), the space of tempered distributions on \(\rone\). We shall write \( \hat{f}\) interchangeably with \(\mathcal F(f)\). The inverse of \(\mathcal F\) with this normalisation is then given by \(f(x)=\f{1}{2\pi} \int_{\rone} \hat{f}(\xi) e^{ i x\xi} \dxi\).

\subsection{The family of kernels $\{K_\alpha\}_\alpha$} 
The kernel \(K\) in \eqref{eq:whitham} arises from the linear dispersion relation in the free-boundary Euler equations \cite{MR3060183}, where its symbol \(m(\xi)\) describes the dependence of wave speed of a travelling wave-train upon its frequency \(\xi\). Since \(m\) is real and even, the operator \(L = K \ast\) has a well-defined square root given by \(\mathcal F (\sqrt{L}f) =  \sqrt{m(\cdot)} \mathcal{F} f\). More generally, every symbol
\begin{equation}\label{eq:K_alpha}
\widehat{K_\al}(\xi)=\left(\f{\tanh(\xi)}{\xi}\right)^{\al}, \qquad \alpha \in (0,1),
\end{equation}
gives rise to a corresponding operator \(L_\alpha = K_\alpha \ast\), for which one has the following result.

\begin{lemma}\label{lemma:K_alpha}
The kernel \(K_\alpha\) in \eqref{eq:K_alpha} satisfies the following properties.
\begin{itemize}
\item[(i)] $K_\alpha$ is smooth outside the origin, and rapidly decaying: for each fixed \(\alpha \in (0,1)\) and $N \geq 1$ one has
$$
K_\al(x)\lesssim x^{-N}, \qquad x>0.
$$
\item[(ii)] $K_\alpha$ is bell-shaped and strictly convex.
\item[(iii)] $K_\alpha$ has unit operator norm: $\|K_\alpha \|_{L^1}=\int_{\rone} K_\al(x) \dx=\widehat{K_\al}(0)=1$. 
\item[(iv)] $K_\alpha$ belongs to  \(L^p(\rone)\) exactly for \( p<1/(1-\alpha)\). In particular, \(K_{\frac{1}{4}}\) is in \(L^{4/3-}(\rone)\).
\end{itemize}
\end{lemma}

\begin{remark}
There are several approaches towards the properties of  \(K_\alpha\). We choose here to give an elementary and direct proof. A more subtle method yielding somewhat more information about the kernels can be found in \cite{MR4002168}.
\end{remark}

\begin{proof}
(i) To prove the decay and smoothness of $K_\al$, we write 
\begin{equation}\label{eq:varrho}
\begin{aligned}
K_\al(x) &= \f{1}{2\pi} \int_{\rone} \left(\f{\tanh(\xi)}{\xi}\right)^{\al} \varrho(\xi) e^{ i x\xi} \dxi\\ 
&\quad+
 \f{1}{2\pi} \int_{\rone} \left(\f{\tanh(\xi)}{\xi}\right)^{\al} (1-\varrho(\xi)) e^{ i x\xi} \dxi, \\
&= K_\al^1(x)+K_\al^2(x), 
\end{aligned}
\end{equation}
where $\varrho \in C_0^\infty(\rone)$ is even, supported in $(-\frac{1}{2},\frac{1}{2})$, and satisfies \(\varrho(\xi)=1\) for \(|\xi|<\frac{1}{4}\). Since the function $\xi\to \f{\tanh(\xi)}{\xi}$ is smooth (in fact, real analytic) and strictly positive, we have that  $\widehat{K_\al^1}$ is a $C^\infty_0$-function, whence $K_\al^1$ is a Schwartz function.  

$K_\al^2$ can be dealt with using integration by parts. Taking the form of the multiplier \(m\) into account, one readily sees that 
\begin{equation}
\label{30}
K_\al^2(x)=x^{-l} \int_{\rone} m_{l, \al}(\xi) e^{ i x\xi} \dxi, \qquad x >0,
\end{equation}
where $m_{l, \al}$ is $C^\infty$ function with $|m_{l, \al}(\xi)|\lesssim_{l,\al} |\xi|^{-l-\al}$, \(l\) being an arbitrary positive integer.  This proves the super-polynomial decay rate for $K_\al$. Also, differentiating under the integral sign in \eqref{30} yields that \(K_\al^2\) is  of class \(C^{l-1}\) outside the origin. Since \(l\) is arbitrary, we conclude that $K_\al$ is in fact smooth on the same domain.

(ii) Note that the integral obtained by differentiating under the integral sign as in
\[
\Diff_x K(x) = - \frac{2}{\pi} \int_0^\infty (\xi \tanh(\xi))^{1/2} \sin(x\xi)\dxi,
\]
is not well-defined. Thus, even a formal argument is hard to invoke for the one-sided monotonicity of \(K\). To circumvent this difficulty, we consider instead of \(\Diff_x K_\alpha (x)\) the product \(-x \Diff_x K_\alpha(x)\). This is an element of \(S^\prime\), and we show that its Fourier transform is positive definite. Hence, \(x \Diff_x K_\alpha(x) < 0\) for all \(x \neq 0\) and all  \(\alpha \in (0,1]\).

We carry out the calculation first for the Whitham kernel (the case \(\alpha = \frac{1}{2}\)). Relying on Bochner's Theorem \cite{MR2284176}, we want to prove that
\[
\Diff_\xi (\xi m(\xi)) = -\F (x \Diff_x K(x)) 
\] 
is absolutely continuous; then \(x \Diff_x K(x) = - \F^{-1} ( \Diff_\xi (\xi m(\xi)))\) will be negative. Notice that \(\Diff_\xi (\xi m(\xi))\) is even in \(\xi\), so that it is enough to consider non-negative \(\xi\) in the calculations to come. Now,
\begin{equation}\label{eq:ximxi}
\begin{aligned}
\Diff_\xi (\xi m(\xi))
&= \frac{1}{2} \frac{\xi}{\sqrt{\xi \tanh(\xi)}} \left( \frac{\tanh(\xi)}{\xi} + \frac{1}{\cosh^2(\xi)}\right)\\
&= \frac{1}{2}m(\xi)  \left( 1 + \frac{\xi}{\tanh(\xi) \cosh^2(\xi)}\right)\\
&=\frac{1}{2}m(\xi)\left( 1 + \frac{2\xi}{\sinh(2\xi)}\right),
\end{aligned}
\end{equation}
which is positive definite in view of that \([\xi \mapsto 1]\), \([\xi \mapsto \xi/\sinh(\xi)]\) and \(m\) all are (for the latter facts, see \cite{MR2284176}).

For a general \(\alpha \in (0,1)\), a calculation like \eqref{eq:ximxi} shows that
\[
\Diff_{\xi} \left(\xi m(\xi) \right)^{2\alpha} = \left( m(\xi) \right)^{2\alpha} \left( 1- \alpha + \alpha \frac{2\xi}{\sinh(2\xi)}\right),
\]
which is positive definite for \(0 \leq \alpha \leq 1\), in view of that \(\xi/\sinh(\xi)\) and \(1-\alpha\) are, and that \(m\) is infinitely divisible, meaning that any positive power of it is positive definite.

To prove the convexity of \(K_\alpha\), one can use theory for Stieltjes and completely monotone functions \cite{MR4002168}. In the case \(\alpha \leq \frac{1}{2}\), it is however possible with a more straightforward approach. As above, one calculates
\begin{equation}\label{eq:convexgeneral}
\begin{aligned}
&\F \left( x^2 \Diff_x^2 \F^{-1}\left(m(\xi)^{2\alpha}\right)\right)\\
&= m(\xi)^{2\alpha} \bigg( (1-\alpha)(2-\alpha) +  \frac{2\alpha (2-\alpha) \xi}{\sinh(\xi) \cosh(\xi)} +  \frac{\alpha^2 \xi^2}{\sinh^2(\xi)\cosh^2(\xi)}\\ 
&\quad- \frac{\alpha \xi^2}{\sinh^2(\xi)} - \frac{\alpha \xi^2}{\cosh^2(\xi)} \bigg)\\
&= m(\xi)^{2\alpha} \bigg( (1-\alpha)(2-\alpha)\\ 
&+  \frac{\alpha \xi}{\sinh(\xi)} \left( \frac{\alpha \xi}{\sinh(\xi)} \frac{1}{\cosh^2(\xi)} + \frac{\nu}{\cosh(\xi)} - \frac{\xi}{\sinh(\xi)}\right)_1 \\
&\quad+ \frac{\alpha}{\cosh(\xi)} \left( \frac{(4-2\alpha - \nu)\xi}{\sinh(\xi)} -  \frac{\xi^2}{\cosh(\xi)}\right)_2 \bigg),
\end{aligned}
\end{equation}
where \(\nu\) is an arbitrary real number. The Fourier transforms of all factors appearing in the outmost parenthesis are explicitly known \cite{MR1055360, MR942661}, and one can find a value \(\nu\) such that the entire expression is positive definite.\footnote{In particular, \(\frac{\pi}{2} + 2\alpha\sqrt{\frac{2}{\pi}}\left( \frac{2}{\pi}-1 \right) \leq \nu \leq 4 - 2\alpha - \frac{\pi}{2}\) guarantees that both the parentheses \(( )_1\) and \(( )_2\) are positive definite.}  Since \(m(\xi)\) is infinitely divisible, the convexity of the kernel away from the origin is therefore guaranteed (in the case \(\alpha \leq 1/2\), one can see directly from the explicit Fourier transforms that the convexity is strict).

(iii) The proof of this is directly given in the statement.

(iv) We deduct the homogeneous part of \(\hat K_\alpha\). Write
\[
\hat K_\alpha(\xi) = \left(\frac{\tanh \xi}{\xi}\right)^\alpha = \frac{1}{|\xi|^\alpha}  + \frac{(\tanh|\xi|)^\alpha -1}{|\xi|^{\alpha}},
\]
which is valid since \(\hat K_\alpha\) is even. Then \(K_\alpha(x) \eqsim  \frac{1}{|x|^{1-\alpha}}  + \F(\frac{(\tanh|\xi|)^\alpha -1}{|\xi|^{\alpha}})(x)\). The first part is clearly in \(L^p_\text{loc}(\rone)\) for \(1 \leq p < 1/(1-\alpha)\), and for no greater \(p\). The second part is the Fourier transform of an \(L^1\)-function with exponential decay, and hence smooth by a version of Schwartz's (Paley--Wiener) theorem \cite{MR52555}. Since, by (i), \(K_\alpha\) has rapid decay, \(K_\alpha \in L^{p}(\rone)\) exactly when \(p < 1/(1-\alpha)\). 
\end{proof}

As above, we shall use the convention that $K = \Kt$. Recall that \(L_\alpha\) is the operator with symbol \(\hat K_\alpha\), defined for \(\alpha \in (0,1)\).

\begin{lemma}\label{lemma:convolution_estimates}
The operator \(L_\alpha\) is an isomorphism $L^2(\rone) \to H^\al(\rone) $, with its inverse given by the symbol \(\hat K_{-\alpha}\).
  For fixed \(\alpha\) and \(1 \leq q\leq p<\infty\) satisfying  \(\f{1}{q}-\f{1}{p}\leq \al\), we have 
\[
\|K_\al*f\|_{L^p}\lesssim  \|f\|_{L^q},
\]
and for fixed $q>1/\al$, 
$$
\|K_\al*f\|_{L^\infty}\lesssim \|f\|_{L^q}.
$$
\end{lemma}

\begin{proof}
To see that \(L_\alpha \colon L^2(\rone) \to H^\al(\rone)\) is an isomorphism, note that 
\(
(\f{\xi}{\sqrt{1+\xi^2}\tanh(\xi)})^{\al}
\)
is an  isomorphism on $L^2$, since it is a positive and bounded function which is also bounded away from the origin. 

The \(L^q\)--\(L^p\)-estimate is almost a direct consequence of Young's inequality, which states that
\[
\| f * g \|_{L^p}  \leq   \| f \|_{L^q} \| g \|_{L^r}, \quad\text{ when }\quad 		\frac{1}{r} + \frac{1}{q} = 1 + \frac{1}{p}, 
\]
and \(1 \leq q,p,r \leq \infty\). In view of that \(K_\alpha \in L^r\) whenever \(1/r \in (1-\alpha,1]\), the statement follows for \(\f{1}{q}-\f{1}{p}< \alpha\). The same argument for \(p = \infty\) gives the continuity \(L^q \to L^\infty\) of \(L_\alpha\) whenever \(q > 1/\alpha\). For the case \(\frac{1}{q}-\frac{1}{p} = \alpha\), one must use the generalisation of Young's inequality to the weak \(L^{r, \infty}\)-space, namely
\[
\| K_a * f \|_{L^p}  \leq  \| K_a \|_{L^{r, \infty}}  \| f \|_{L^q},
\]
which is valid for the same relation between \(p\), \(q\) and \(r\) (although not for $p=\infty$ or $q=1$, see for example \cite[Thm 1.4.24]{MR2445437}). Since \( |x|^{\alpha-1} \in {L^{1/(\alpha-1), \infty}}\), it follows from the formula \(K_\alpha(x) \eqsim  |x|^{\alpha - 1}  + \tilde f(x)\), \(\tilde f \in C^\infty(\rone)\), deduced in the proof of Lemma~\ref{lemma:K_alpha}~(iv), and the  decay of \(K_\alpha\), that also \(K_\alpha \in {L^{r,\infty}}\). This proves the continuity \(L^q \to L^p\) of \(L_\alpha\) for \(\f{1}{q}-\f{1}{p} = \al\).
\end{proof}

\begin{lemma}
\label{le:2.4}
\(L_\alpha\) preserves bell-shapedness, and the maximum of \(L_\alpha f\) for any non-constant bell-shaped function \(f\) is attained only at the origin.
\end{lemma}

\begin{proof}
Recall that \(L_\alpha\) acts by convolution with \(K_\alpha\), which according to Lemma~\ref{lemma:K_alpha} is itself bell-shaped. The convolution of two even and positive functions is clearly even and positive, so it remains to show that \(K_\alpha \ast f\) is decreasing on the positive half-axis if \(f \in L^p(\rone)\) is bell-shaped, where \(p\) is some number in \([1,\infty]\). 
Consider the difference
\begin{align*}
&K_\alpha \ast f(x + h) - K_\alpha \ast f(x)\\ 
&\quad= \int_0^{\infty} \textstyle \left( K_\alpha(z+x+\frac{h}{2}) - K_\alpha(z-x-\frac{h}{2}) \right) \left( f(z - \frac{h}{2}) - f(z + \frac{h}{2}) \right)\dz,
\end{align*}
where we have used the evenness of \(K_\alpha\) and \(f\) to rewrite the integral.  For \(x > 0\) and \(0 < h \ll 1\) the factors in the integrand have differing signs, whence the first assertion follows.

To see that the maximum of \(L_\alpha f\) is attained only at one point, fix $x_0 > 0$  and consider $g_{x_0}(y):=f(x_0+y)$. By  in \cite[Thm 3.4]{0966.26002}, one has
\begin{eqnarray*}
& & K*f(x_0) = \int K(y) f(x_0-y) \dy=\int K(y) f(x_0+y) \dy= \\
&=& \int K(y) g_{x_0}(y) \dy\leq  \int K^*(y) (g_{x_0})^*(y)  \dy= \int K(y) f(y) \dy= K*f(0).
\end{eqnarray*}
where we have used that the decreasing rearrangement of a translate of a bell-shaped function coincides with the function itself:  
\[
(g_x)^*=(g_0)^*=f^*=f. 
\]
Thus the function $x \mapsto K*f(x)$ achieves its maximum at $x=0$. Moreover, since $K$ is strictly symmetric decreasing, equality is possible (again, according to \cite[Thm 3.4]{0966.26002}) only when $g_{x_0}=g_{x_0}^*$. But that would imply 
$f(x_0+y)=f(y)$ for all $y$, meaning that \(f\) is everywhere constant.
\end{proof}

\subsection{Orlicz spaces} 
The following subsection is a short introduction to Orlicz spaces. The reader will find here all the basic results used in this paper.

\subsubsection{Distribution functions and convolution rearrangement inequalities}
For a measurable function $f:\rone\to \rone$, let 
$$
d_f(\al):=|\{x\in \rone: |f(x)|>\al\}|
$$
be its \emph{distribution function}. For every $\vp\in C^1(\rone)$  with \(\varphi(0) = 0\), one has (cf. \cite[Eq. (1.1.7)]{MR2445437}) {\it the layer cake representation formula} 
\begin{equation}\label{150}
\int_{\rone} \vp(|f(x)|) \dx= \int_0^\infty \vp'(\al) d_f(\al)\,\mathrm{d}\al. 
\end{equation}
The \emph{non-increasing rearrangement}, \(f^*\), of the function $f$ is the inverse function of $d_f$, provided that $\al\to d_f(\al)$ is strictly decreasing. In general, we define  $f^*:\rone_+\to\rone_+$ by
 $$
 f^*(t)=\inf\{s>0: d_f(s)\leq t\}, \qquad t \geq 0. 
 $$ 
Then $f^*$ is a non-increasing function, and one furthermore has that $d_{f^*}(\al)=d_f(\al)$.  
Let $f^\#(t)=f^*(2|t|)$, \(t \in \rone\). Then $f^\#$ provides us with a convenient way of characterizing bell-shapedness, namely, $f$ is bell-shaped if and only if $f^\# =f$.

The function $f^\#$ is equidistributed with $f$, in the sense that 
 $d_{f^\#}(\al)=d_f(\al)$ for all $\al>0$. Moreover,
 \begin{equation}
 \label{190}
 \supp\ f\subset [-L,L] \Longrightarrow \supp\ f^\#\subset [-L,L].
 \end{equation}
 Indeed, for  $t>L$, we have 
 \begin{align*}
 0 \leq f^*(2 t) &=\inf\{s>0: |\{x: |f(x)|>s\}|\leq 2 t\}\\
&\leq \inf\{s>0: |\{x: |f(x)|>s\}|\leq 2 L\} =0,
 \end{align*}
 since $|\{x: |f(x)|>s\}|\leq |\supp f|\leq 2 L$ for any $s>0$. Thus, for 
 $|t|>L $, we have $f^\#(t)=f^*(2t)=0$, hence $\supp\ f^\#\subset [-L,L]$. 
 
A result that will be essential to our investigation is the \emph{Riesz convolution--rearrangement inequality}. This has later been generalized in several ways, see, e.g.,~\cite{MR1604163, MR1574064}).

\begin{lemma}[Riesz]
Let  \(f\), \(g\) and \(h\) be measurable real functions. Then
 \begin{equation}
\label{160}
\int_{\rtwo}f(y) g(x-y) h(x) \dx\dy  \leq  \int_{\rtwo}  f^\#(y) g^\#(x-y) h^\#(x) \dx \dy.
\end{equation}
\end{lemma}
\begin{remark}
The integrals in \eqref{160} are interpreted in the extended sense that they may take infinite values. In particular, the integral on the left-hand is finite whenever the right-hand side is finite.
\end{remark}

\subsubsection{Young's functions}
We mainly follow the exposition in the book of Rao and Ren \cite{MR1113700}, to which we refer the reader for further details. 

We shall say that a function $\Psi: \rone\to \rone_+$ is a \emph{Young function}, if it is even, convex,  
 and satisfies \(\Psi(0) = 0\),  $\lim_{x\to \infty} \Psi(x)=\infty$. In the literature, a continuous Young function that satisfies 
 \begin{equation}\label{eq:N_function}
 \Psi(x) = 0 \Longleftrightarrow x = 0,  \qquad \lim_{x\to 0} \f{\Psi(x)}{x}=0,  \quad\text{ and }\quad \lim_{x\to \infty} \f{\Psi(x)}{x}=\infty, 
 \end{equation}
 is sometimes called a nice Young function (or an $N$-function). \emph{We shall, however, only work with functions satisfying \eqref{eq:N_function}, and will therefore make these requirements also for Young functions.}

\subsubsection{Orlicz spaces} 
Orlicz spaces are function spaces built on a Young function \(\Psi\). More specifically, for a Borel measure\footnote{In our applications, the measure $\mu$ will exclusively be the Lebesgue measure on $\rone$, or its restriction to a finite interval of the form $(-a,a)$.}  $\mu$ on $\rn$, consider
\begin{equation}\label{eq:Orlicz space}
\cl^\Psi = \{f:\rn\to {\mathbf R}: \int_{\rn} \Psi(\al f(x)) \,\mathrm{d}\mu(x)<\infty\ \ \textup{for some}\ \ \al>0\}.
\end{equation}
Identifying \(\alpha\) with  \(1/\lambda\), one may define a norm via the so-called Minkowski gauge functional, 
\begin{equation}\label{eq:gauge norm}
N_\Psi(f)=\inf\{\la>0: \int_{\rn} \Psi\left(\f{f(x)}{\la}\right) \,\mathrm{d}\mu(x)\leq 1\}. 
\end{equation}
With this definition the pair $(\cl^\Psi, N_\Psi)$ becomes a Banach space, called the \emph{Orlicz space for \(\Psi\)}.  The norm $N_\Psi$ is usually referred to as the \emph{gauge norm}. A Young function $\Psi$ is said to satisfy a \emph{global $\De_2$-condition} ($\Psi\in \De_2$, for short) if one has
\begin{equation}\label{eq:delta_2}
\Psi(2x)\lesssim \Psi(x). 
\end{equation}
for all $x\geq  0$. In the case when $\Psi$ is strictly increasing, continuous and satisfies a $\De_2$-condition, one may alternatively take 
\begin{equation}\label{eq:gauge norm 2}
\int_{\rn} \Psi\left(\f{f(x)}{N_\Psi(f)}\right) \dx= 1
\end{equation}
as a definition of the gauge norm \(N_\Psi\). The following result is adapted from \cite[p. 280]{MR1113700} and shows that under additional regularity assumptions on the Young function \(\Psi\), the gauge norm has directional derivatives.

\begin{lemma}{\cite{MR1113700}}
\label{lemma:Gateaux} 
Let $\Psi$ be a differentiable and strictly convex Young function with a strictly positive derivative on the positive half-axis. Then $N_{\Psi}(\cdot)$ is Gateaux differentiable, with derivative 
$$
\f{d}{d\varepsilon} N_{\Psi}(f_0+\varepsilon h)|_{\varepsilon=0} =\f{\dpr{h}{\Psi'(f_0)}}{\dpr{f_0}{\Psi'(f_0)}},
$$
whenever $N_{\Psi}(f_0)=1$. 
\end{lemma}

The next lemma relates the gauge norm to the distribution function. We shall make use of this in our construction of bell-shaped solutions. 
 \begin{lemma} \label{le:80}
 Let $\Psi$ be a $C^1$-Young function and let $f\in \cl^\Psi$. Then
 $$
 N_{\Psi}(f)=N_{\Psi}(f^*)=N_{\Psi}(f^\#).
 $$
 In particular, $f^*$ and \(f^\#\) both belong to \(\cl^\Psi\).
 \end{lemma}

 \begin{proof}
 By rescaling, we can assume that $N_{\Psi}(f)=1$, so that $ \int_{\rone} \Psi(|f(x)|) \dx=1$. We need to show that $N_{\Psi}(f^*)=N_{\Psi}(f^\#)=1$.  According to \eqref{150} and the fact that $d_{f^*}(\al)=d_{f^\#}(\al)=d_{f}(\al)$, we have that 
\begin{align*}
 1=\int_{\rone} \Psi(|f(x)|) \dx &= \int_0^\infty \Psi'(\al) d_f(\al) \,\mathrm{d}\al\\
 &= \int_0^\infty \Psi'(\al) d_{f^\#}(\al) \,\mathrm{d}\al = \int_{\rone} \Psi(f^\#(t)) \,\mathrm{d}t.
\end{align*}
It follows that   $N_{\Psi}(f^\#)=1$. The argument for $N_{\Psi}(f^*)=1$ is analogous. 
 \end{proof}

\section{The variational problem}\label{sec:variational}
In order to construct a solution to \eqref{eq:steady_whitham}, one naturally considers a constrained optimization problem connected to it. Formally, if \(f\) is a maximizer of
\begin{equation}\label{eq:J}
\J(f) = \|\Kf*f\|_{L^2}
\end{equation}
under the constraint that
\begin{equation}\label{eq:constraint1}
\I(f) = \int \left(  f^2(x)  - {\textstyle \frac{1}{3}} f^3(x) \right) \dx=1, 
\end{equation}
one obtains a solitary solution of the steady Whitham equation \eqref{eq:steady_whitham} by a rescaling argument (the wave speed arises from a Lagrange multiplier principle). The above problem, however, is not well-posed, since  there is no finite supremum of \eqref{eq:J} over functions fulfilling \eqref{eq:constraint1}. One way to see this is by taking a dilated and scaled characteristic function \(\chi_N\) supported on the interval \([-N,N]\) that satisfies \(\I(\chi_N) = 0\). Adding a Schwartz function \(\phi_N(x):=\phi(\cdot-N)\), where $\phi$ is supported in $(0,1)$,   and fulfilling \(\I(\phi_N) = \I(\phi)=1\), the sum fulfils the constraint \(\I(\chi_N + \phi_N) = 1\). The quadratic energy \(\J(\chi_N + \phi_N)\) on the other hand scales like \(\sqrt{N}\), as the dilation in \(\chi_N\) contributes \(N\) to the functional, but the translation in \(\phi_N\) results only in a phase-shift. Therefore
\[
		\J(\chi_N + \phi_N))^2 = \|\Kf*\chi_N\|^2+\dpr{\Kf*\chi_N}{\Kf*\phi_N}+\|\Kf*\phi_N\|^2
\]
is for large \(N\) dominated by
\[
	\|\Kf*\chi(\cdot/N)\|^2 =  4 \int \left| \f{\tanh(\xi)}{\xi}\right|^{\f{1}{2}} 
	\f{\sin^2(N\xi)}{\xi^2}d\xi \eqsim N
\]
by Plancherel's theorem. 

One way to remedy this is to consider functions that additionally satisfies
\begin{equation}\label{eq:constraint2}
\sup f \leq 1.
\end{equation}
As we shall show, the resulting problem is solvable and yields a fairly rich family of solutions of the Whitham equation. Technically, though, \eqref{eq:constraint1} poses challenges as it does not describe a function space. One way of dealing with this problem is to work in a ball in a regular Sobolev space \(H^s(\rone)\), for which both \eqref{eq:constraint1} and \eqref{eq:constraint2} may be fulfilled. Variants of this approach have been used in \cite{EGW11,GW10} and other investigations, and yields small-amplitude solutions of the original problem.

\subsection{The Orlicz space of constraints}
 
Our approach does not per se rely on smallness. As is often done, we deal with the loss of compactness on \(\rone\) by considering a sequence of problems on increasing intervals. To handle~\eqref{eq:constraint1} and~\eqref{eq:constraint2}, we enlarge the set of functions allowed for by considering the Orlicz function
\begin{equation}\label{eq:Psi}
\Psi(f)=
\begin{cases}
\alpha f^2-\frac{1}{3} f^3, \qquad & 0\leq f < \alpha, \\
\f{2}{3} \alpha^3 + \alpha^2 (f-\alpha)+ (f-\alpha)^3, \qquad\quad &  \phantom{0 \geq \,} f \geq  \alpha,
\end{cases}
\end{equation}
for fixed and positive values of \(\alpha > 0\). By varying $\alpha>0$, we obtain a non-trivial family of waves, and we show that the relevant waves exist in an interval of the form $\al\in [\al_0, \infty)$, with $\al_0$ to be appropriately defined later.
Generally speaking, small waves of the Whitham equation \eqref{eq:steady_whitham} correspond to maximizers with large $\al$ and vice versa.

\begin{lemma}\label{lemma:L2L3}
	The function \(\Psi\) defined by \eqref{eq:Psi} is a strictly convex, strictly increasing, \(C^2\)-Young function for which~\eqref{eq:gauge norm 2} defines a norm. For any fixed value of \(\alpha > 0\), the corresponding Orlicz space \(\mathcal{L}^{\Psi}\) satisfies
	\[
	\mathcal{L}^{\Psi} \cong L^2(\rone) \cap L^3(\rone), 
	\]
	in the sense that
	\begin{equation}\label{eq:norm equivalence}
	N_{\Psi}(f) \eqsim \max ( \|f\|_{L^2}, \|f\|_{L^3} ).
	\end{equation}
	The estimates \(\|f\|_{L^2(\rone)} \lesssim \frac{1}{\sqrt{\alpha}} N_\Psi(f)\) and \(\|f\|_{L^3(\rone)} \lesssim N_\Psi(f)\) are furthermore uniform in \(\alpha>0\).
\end{lemma}

\begin{proof}
	The function \(\Psi\) is of class \(C^2\) by construction, with \(\Psi^\prime(f) > 0 \) for \(f > 0\) and \(\Psi^{\prime\prime}(f) > 0\) for all \(f \neq \pm \alpha\). Hence, \(\Psi\) is strictly increasing and convex in the sense of a Young function. To see that it is indeed Young, note that~\eqref{eq:N_function} trivially holds. Similarly, the \(\Delta_2\)-condition~\eqref{eq:delta_2} can be easily checked. Then \eqref{eq:Orlicz space} defines an Orlicz space with norm given, equivalently, by~\eqref{eq:gauge norm} and~\eqref{eq:gauge norm 2}.
	
	To prove that this Orlicz space is, for given \(\alpha > 0\), isomorphic to \(L^2(\rone) \cap L^3(\rone)\) in the above sense, pick first  \(f\) with $N_{\Psi}(f)=1$. Then \(1 = \int_\rone \Psi(f) \dx\). For \(0 \leq f \leq \alpha\) we have that
	\(
	\frac{2 \alpha}{3}  f^2 \leq \alpha f^2 - \frac{1}{3}f^3 \leq \alpha f^2,
	\)
	and for \(f > \alpha\) that
	\(
	\frac{2 \alpha^3}{3} + (f-\alpha)^3 \leq \frac{2 \alpha^3}{3} + \alpha^2(f-\alpha) + (f-\alpha)^3 \leq 3 f^3.
	\)
	A simple consideration of \(f \gtrless 2\alpha\) shows that the left-hand side is uniformly bounded from below by \(\frac{1}{12} f^3\). Combining these estimates we obtain 
	\[
	\int (\alpha f^2 + f^3) \dx \lesssim N_{\Psi}(f) \lesssim  \int (\alpha f^2 + f^3) \dx,
	\]
	uniformly in \(\alpha > 0\) for \(N_{\Psi}(f) = 1\). Since for such \(f\)  one has \((N_{\Psi}(f))^p = 1\) for all \(p\), a rescaling argument yields \eqref{eq:norm equivalence}.
\end{proof}

Relative to $\Psi$, we now consider the problem of maximizing \(\J(f)\) under the constraint that
\begin{equation}\label{eq:real constraint}
N_{\Psi}(f)=1,
\end{equation}
for a given positive value of \(\alpha\). Note that by scaling, the maximizers of ~\eqref{eq:J}, if any, under the constraint \eqref{eq:real constraint} are the same if one enlarges the constraint to include the whole closed ball $\{N_\Psi(f)\leq 1\}$ in \(\mathcal{L}^\Psi\). To handle compactness, we consider first local versions of this maximization problem.

\subsection{A family of local problems}
For the purpose of obtaining a convergent subsequence, we consider  functions $f$ supported on an interval $[-2^l, 2^l]$, \(l \in \N\). More precisely, in this section we find functions \(f = f_l\) that realize the maximum of \(\J(f)\) under the constraint that \(N_{\Psi}(f)=1\), and that additionally satisfy
\begin{equation}
\textup{supp}\ f\subset [-2^l, 2^l].
\end{equation}
The value of the parameter \(\alpha > 0\) will for now be held constant. Define
\[
J = \sup_{N_{\Psi}(f)=1} \|\Kf*f\|_{L^2}, 
\]
so that \(J = \max\J\) under the constraint \(N_{\Psi}(f)=1\) when a maximizer exists, and similarly
\begin{equation}\label{eq:J_l}
J_l = \sup\limits_{\substack{N_{\Psi}(f)=1 \\ \supp f\subset [-2^l, 2^l]}} \|\Kf*f\|_{L^2}, \qquad l \in \N.
\end{equation}
Both \(J\) and \(J_l\) depend on \(\alpha\).

\begin{lemma}\label{lemma:J_l limit}
For any fixed value of \(\alpha\), one has 
\begin{equation}\label{gn:50}
\lim_{l\to \infty} J_l = J \lesssim \frac{1}{\sqrt{\alpha}},
\end{equation} 
where the latter bound is uniform in $\alpha$.
\end{lemma}

\begin{proof}
We already know that $J_{l}\leq J$. From Hausdorff--Young's inequality and Lemma~\ref{lemma:L2L3}, 
$$
\|\Kf*f\|_{L^2}\leq \|\Kf\|_{L^1} \|f\|_{L^2}\lesssim \alpha^{-1/2} N_{\Psi}(f) = \alpha^{-1/2},
$$
whenever $f$ satisfies the constraint. Thus, $J \lesssim \alpha^{-1/2}$, uniformly in $\alpha$. 

Also, it is clear from the definition that $\{J_{l}\}_l$ is an increasing  sequence. We will now show that \eqref{gn:50} holds. Indeed, let $\varepsilon>0$. Then there exists $f^\varepsilon$ with $N_\Psi(f^\varepsilon)=1$ such that 
$$
\|\Kf*f^\varepsilon\|_{L^2}>J-\varepsilon.
$$
For this function $f^\varepsilon\in \cl^{\Psi}\hookrightarrow L^2$ there is also an 
$l$ so that $\|f^\varepsilon \chi_{|x|>2^l}\|_{L^2}<\varepsilon$. Note that since $\cl^\Psi$ is a lattice,  
$N_\Psi(f^\varepsilon \chi_{|x|<2^l})\leq 1$. 
 Thus, 
\begin{align*}
J_l &\geq \|\Kf*(f^\varepsilon\chi_{|x|<2^l})\|_{L^2} \geq \|\Kf*f^\varepsilon \|_{L^2}-
 \|\Kf*(f^\varepsilon\chi_{|x|>2^l}) \|_{L^2}\\
 &\geq J - 2\varepsilon,
\end{align*}
where in the last inequality, we have  estimated
$$
\|\Kf* (f^\varepsilon\chi_{|x|>2^l} ) \|_{L^2}\leq  \|\Kf\|_{L^1} 
\|f^\varepsilon\chi_{|x|>2^l}\|_{L^2}\leq \varepsilon.
$$
Thus \eqref{gn:50} holds. 
\end{proof}

Our next lemma establishes a lower bound on $J$, consistent with the estimate \(J \lesssim \alpha^{-1/2}\) from Lemma~\ref{lemma:J_l limit}. Most of the time, we shall just need a simple corollary of the below result, namely that
$
J \geq  \alpha^{-1/2}. 
$

\begin{lemma}
\label{lemma:J-bound}
There exists an absolute and positive constant $c_0$ such that
\begin{equation}\label{eq:J_l bound}
J_l \geq \frac{1}{\sqrt{\alpha}} \left(1+\frac{c_0}{1 +\alpha^2}\right)
\end{equation}
for all $\alpha>0$ and all \(l \gtrsim |\log(\alpha)|\). In particular,
\(
J - \frac{1}{\sqrt{\alpha}}
\)
is positive for any fixed \(\alpha\).
\end{lemma}

\begin{proof}
The exists a positive constant \(c >0\) such that \((\frac{\tanh(\xi)}{\xi})^{1/2} \geq 1 - c \xi^2\), for all \(\xi \in \rone\). Hence, 
\begin{align*}
\int \left(\frac{\tanh(\xi)}{\xi}\right)^{1/2} |\hat f(\xi)|^2 \dxi &\geq \int |\hat f(\xi)|^2 \dxi - c \int | \xi \hat f(\xi)|^2 \dxi,
\end{align*}
with \(c\) independent of \(f\). Letting \(\supp(f) \subset [-2^l,2^l]\) we obtain the lower estimate
\begin{align}\label{eq:lower J_l}
J_l^2 &\geq \|f\|_{L^2}^2 - c \|f'\|_{L^2}^2 = {\textstyle \frac{1}{\alpha}} \left(1 + {\textstyle \frac{1}{3}} \|f\|_{L^3}^3 \right) - c \|f'\|_{L^2}^2,
\end{align}
for any function \(f \leq \alpha\) satisfying the constraint \(N_\Psi(f) = 1\). To eliminate the \(\alpha\)-dependence in \(\int \Psi(f) \dx = 1\), introduce a smooth test function \(q\) with \(\supp(q) \subset [-1,1]\), \(0 \leq q \leq 1\), and consider
\[
f(x) = \alpha q(\alpha^3 x).
\]
Then \(N_\Psi(f) = 1 \) if and only if \( \int q^2(x) - \frac{1}{3} q^3(x) \dx = 1\), and it is evident that we can find  \(q\) satisfying this additional assumption. Replacing \(f\) in \eqref{eq:lower J_l} by \(f = \alpha q(\alpha^3 \cdot)\), we find that
\begin{equation}\label{eq:lower J alpha small}
J_l^2 \geq  {\textstyle \frac{1}{\alpha}} \left(1 + {\textstyle \frac{1}{3}} \|q\|_{L^3}^3 \right) - c \alpha^5 \|q'\|_{L^2}^2,
\end{equation}
provided that \(\alpha^3 2^l \geq 1\), to satisfy the constraint that \(\supp(f) \subset [-2^l,2^l]\).
Thus, there exists a positive constant \(c_0\) such that \(J_l^2 \geq  {\textstyle \frac{1}{\alpha}} \left(1 + c_0 \right)\) whenever \(2^{-l/3} \leq \alpha \leq \alpha_0\), where \(\alpha_0\) is independent of \(l\).

For the general case of \(\alpha \in (0,\infty)\), introduce the additional scaling
\[
\tilde q(x) = \mu q(\mu^2 \eta x),  
\]
where \(\mu \ll 1\) is small parameter, and \(\eta\) is to satisfy the constraint \(N_\Psi(f) = 1\). Indeed, if we let \(\eta = \|q\|_{L^2}^2 - \frac{\mu}{3}\|q\|_{L^3}^3\), then
\begin{align*}
\int \tilde q^2(x) - \frac{1}{3} \tilde q^3(x)\dx &=
\int \left(\mu q(\mu^2 \eta x)\right)^2 - \frac{1}{3}\left(\mu q(\mu^2 \eta x)\right)^3\dx\\ 
&= \frac{1}{\eta} \int ( q^2(x) - \frac{\mu}{3} q^3(x) )\dx = 1.
\end{align*}
The function \(f(x) = \alpha \tilde q( \alpha^3 x)\) has support in \([-2^l,2^l]\)  if
\[
\alpha^3 \mu^2  \eta 2^l \geq \alpha^3 \mu^2 (1- {\textstyle \frac{\mu}{3}})  \|q\|_{L^2}^2 2^l \geq 1.
\]
On the other hand, \eqref{eq:lower J_l} now becomes
\begin{align*}
J_l^2 &\geq  {\textstyle \frac{1}{\alpha}} \left(1 + {\textstyle \frac{\mu}{3 \eta}} \|q\|_{L^3}^3 \right) - c {\textstyle \frac{\mu^4 \alpha^5}{\eta}} \|q'\|_{L^2}^2,
\end{align*}
with \(\eta \eqsim 1\). Any small enough choice of \(\mu \eqsim \frac{1}{1+ \alpha^{2}}\) yields 
\[
J_l^2 \geq \frac{1}{\alpha} \left(1 + \frac{c_0}{1 + \alpha^2}\right),
\]
whenever \(l \gtrsim |\log(\alpha)|\); thus the estimate includes also the case when \(\alpha\) is small. This lower bound on \(J_l\) is uniform in \(\alpha\) because of how \(f \leq \alpha\) was constructed. Finally, as \(J \geq J_l\) by definition, the same bound is valid for \(J\), independently of~\(l\).
\end{proof}

\begin{lemma}
\label{le:70}
For each \(l\geq 0\) there exists a bell-shaped maximizer \(f_l\) of the local maximization problem fulfilling~\eqref{eq:J_l}.
\end{lemma}
\begin{proof}
We start first with an argument that shows that any function $f$ is at best no better than its rearrangement $f^\#$, as far as the local constrained maximization problem is concerned. Indeed, let \(f\in \cl^\Psi\) satisfy \(N_\Psi(f)=1\) with \(\supp(f) \subset [-2^l, 2^l]\).  By Lemma \ref{le:80}, we then have that $N_\Psi(f^\#) =N_\Psi(f) =1$.  Moreover, due to the bell-shapedness of $\Kf$, and in view of Riesz's convolution--rearrangement inequality~\eqref{160}, we have that for all Schwartz functions $g$, 
\begin{equation}\label{eq:Riesz inequality}
\iint_{\rtwo} \Kf(x-y) f(y) g(x) \dx \dy \leq  \iint_{\rtwo} \Kf(x-y) f^\#(y) g^\#(x) \dx \dy.
\end{equation}
Taking the supremum in \eqref{eq:Riesz inequality} over all functions  \(g\) such that  $\|g\|_{L^2}=1$, we obtain  
$$
\|\Kf*f\|_{L^2}\leq \|\Kf*f^\#\|_{L^2},
$$
in view of that for such \(g\) one also has $\|g^\#\|_{L^2}=1$. The supremum in \eqref{eq:J_l} may thus be considered with respect to only bell-shaped functions. The eventual maximizer will then be bell-shaped as well. 

We now prove that there is a maximizer of \eqref{eq:J_l} in the subspace of bell-shaped functions. 
To that aim, let $\{f_n\}_n$ be a maximizing sequence, that is, a sequence of bell-shaped functions with $N_\Psi(f_n)=1$, \(\supp(f_n) \subset [-2^l, 2^l]\), and $\|\Kf*f_n\|_{L^2}\to J_l$ as \(n \to \infty\). By weak compactness, and up to passing to a subsequence, we may assume that 
\[
f_n \rightharpoonup f_0  \qquad\text{ weakly in }\qquad \cl^\Psi.
\]
Note that $f_0$ is bell-shaped (which can be seen by testing against characteristic functions) and, by the lower semicontinuity of the norm,  
\begin{equation}
	\label{eq:lower semi-cont}
	N_\Psi( f_0) \leq \liminf_{n \to \infty} N_\Psi(f_n)=1.
\end{equation}

Let $g_n=\Kf*f_n$. Then $g_n \in H^{1/4}(\rone)$ by Lemma \ref{lemma:convolution_estimates} and, in fact, 
$\| g_n\|_{H^{1/4}} \eqsim \|f_n\|_{L^2}\lesssim \alpha^{-1/2}N_\Psi(f_n)=\alpha^{-1/2}$. By weak compactness, it again follows that there is a function $g_0 \in H^{1/4}(\rone)$ such that $g_n \rightharpoonup g_0$. By uniqueness of weak limits, $g_0=\Kf*f_0$.  In addition, from the integral representation of $g_n$ and the decay of \(K_{\frac{1}{4}}\) proved in Lemma~\ref{lemma:K_alpha}, for all $|x|>2^{l+1}$ one has 
\begin{align*}
0 <g_n(x) &=\int_{-2^l}^{2^l} \Kf(x-y) f_n(y) \dy \lesssim_N |x|^{-N} \int_{-2^l}^{2^l}   f_n(y) \dy\lesssim_N 2^{\frac{l}{2}} 
|x|^{-N},
\end{align*}
where \(N  \geq 1\) is arbitrary.
It follows that $\{g_n\}_n$ is a compact sequence in $L^2(\rone)$ and hence has a convergent subsequence $\{g_{n_m}\}_m$, that, by uniqueness of limits, converges to $g_0$. In effect,
$$
\|\Kf*f_0\|_{L^2}=\|g_0\|_{L^2}=\lim_{m \to \infty} \|g_{n_m}\|_{L^2}=J_l.
$$ 
By \eqref{eq:lower semi-cont} we have $N_\Psi(f_0)\leq 1$; a strict inequality would contradict the definition of $J_l$ and we conclude that $N_\Psi(f_0)=1$, and $f_0$ is the bell-shaped maximizer sought for. 
\end{proof}

Now that we have established the existence of maximizers\footnote{Recall that we have no proof of uniqueness of these.} of the local problem \eqref{eq:J_l}, let us proceed to derive the corresponding Euler--Lagrange equation.

\begin{lemma}\label{lemma:local Euler}
Every maximizer $f_l$ realizing \eqref{eq:J_l} satisfies the Euler--Lagrange equation
\begin{equation}
	\label{eq:local Euler}
	K* f_l(x)=\f{J_l^2}{\dpr{f_l}{\Psi'(f_l)}} \Psi'(f_l(x)), \quad -2^l<x<2^l.
\end{equation}
In addition, $f_l$ is non-degenerate in $L^2$ and $L^3$; there is a constant $c_0>0$ such that
	\begin{equation}
	\label{eq:L3-nondegeneracy}
	\|f_l\|_{L^3}^3\geq 3\frac{c_0}{1+\alpha^2}.
	\end{equation}
	Also, the estimate
\begin{equation}
	\label{eq:bounds on f_l}
	0<f_l(x)\lesssim \al, \quad -2^l<x<2^l,
\end{equation}
holds uniformly for $\al>0$, and there exists $\al_0$ such that
\[
f_{l}(0)< \al,
\]
for all $\al>\al_0$. These estimates hold uniformly for $l\gtrsim |\log(\al)|$.
\end{lemma}

\begin{proof}
It is straightforward to show that for any non-negative function $f$, we have $f\Psi'(f)\simeq \Psi(f)$ (see also Lemma \ref{lemma:inner product} below) and hence
$$
\dpr{f_l}{\Psi'(f_l)}\eqsim \int \Psi(f_l(x)) \dx=1,
$$
uniformly in $\al>0$ and $l\gtrsim |\log(\al)|$, whence the denominator in \eqref{eq:local Euler} is bounded away from zero. By Lemma~\ref{lemma:Gateaux}, the Gateaux derivative of the function $N_{\Psi}(\cdot)$ when $N_{\Psi}(f_l)=1$ may be determined as
$$
\f{d}{d\varepsilon} N_{\Psi}(f_l+\varepsilon h)|_{\varepsilon=0}=\f{\dpr{h}{\Psi'(f_l)}}{\dpr{f_l}{\Psi'(f_l)}}.
$$
Since $f_l$ is a constrained maximizer, we have 
$$
\|\Kf*(f_l+\varepsilon h)\|_{L^2}^2\leq J_l^2 N_{\Psi}(f_l+\varepsilon h)^2,
$$
for every \(L^2\)-function $h$ with \(\supp\ h\subset[-2^l,2^l]\). Expanding in $\varepsilon$, we obtain 
$$
2\varepsilon \left\langle K*f_l-\f{J_l^2}{\dpr{f_l}{\Psi'(f_l)}} \Psi'(f_l), h \right\rangle+o(\varepsilon)\leq 0.
$$
In view of that \(h\) is free to vary only on the interval \([-2^l,2^l]\), we have established \eqref{eq:local Euler}. Next we prove the non-degeneracy in $L^2$ and $L^3$. As $\|K_\frac{1}{4}\|_{L^1}=1$, we have that
	\begin{equation*}
	\|f_l\|_{L^2}^2\geq \|K_\frac{1}{4}\ast f_l\|_{L^2}^2=J_l^2\geq \frac{1}{\alpha}\left(1+\frac{c_0}{1+\alpha^2}\right),
	\end{equation*}
	where the last inequality follows from Lemma \ref{lemma:J-bound}. For the $L^3$-bound, note that
	\begin{equation*}
	f^3\geq 3(\alpha f^2-\Psi(f)),
	\end{equation*}
	with equality only when $f\leq \alpha$. As $\int \Psi(f_l)\dx=1$, it follows that
	\begin{equation*}
	\|f_l\|_{L^3}^3\geq 3\left(\alpha\|f_l\|_{L^2}^2-1\right)\geq 3\frac{c_0}{1+\alpha^2}.
	\end{equation*}

Now we prove \eqref{eq:bounds on f_l}. Note first that since \(f_l\) is bell-shaped, and \(K\) is everywhere strictly positive, the left-hand side of \eqref{eq:local Euler} cannot vanish unless \(f_l\) is identically zero, which is clearly not the case for a maximizer. Hence, the right-hand side is also non-vanishing, and therefore \(f_l(x) > 0\) for \(x\) in the open interval \((-2^l, 2^l)\); outside this interval, the Euler-Lagrange equation~\eqref{eq:local Euler} does not hold. Recall that $\dpr{f_l}{\Psi'(f_l)}\lesssim 1$ (cf. Lemma \ref{lemma:inner product}) and $J_l > \alpha^{-1/2}$ for \(l \gtrsim |\log(\al)|\) according to Lemma \ref{lemma:J-bound}. 
So wherever $f_l > \alpha$, we have from~\eqref{eq:Psi} and~\eqref{eq:local Euler} that
\begin{equation}
\label{l;}
\alpha^2+3 \left(f_l-\alpha\right)^2 \leq   \f{\langle f_l, \Psi'(f_l) \rangle}{J_l^2} |K*f_l| \lesssim \alpha \|f_l\|_{\infty},
\end{equation}
uniformly in $\alpha>0$, where we have used that $|K*f_l | \leq \|K\|_{L^1} \|f_l\|_{\infty}$. This inequality gives the uniform upper bound on $\|f_l\|_\infty$ in \eqref{eq:bounds on f_l}.
 
Lastly we prove that $f_l(0)<\alpha$ for all $\alpha$ sufficiently large. Assuming that $f_\al(0)>\al$, we have according to \eqref{l;}  evaluated at $x=0$, Lemma \ref{lemma:convolution_estimates} and Lemma \ref{lemma:L2L3} that
$$ 
\al^2\lesssim  \al \|K*f_l\|_{L^\infty}\lesssim \al \|f_l\|_{L^3}\lesssim \al,
$$
uniformly in \(\alpha\) and \(l \gtrsim |\log(\alpha)|\). Clearly, this is a contradiction for all large enough $\al$, whence $f_l(0)<\al$ for all such \(\alpha\) and \(l\).

\end{proof}
In the previous lemma we established that $\|f_l\|_{L^\infty}\lesssim \al$ uniformly in $\al>0$ and $l\gtrsim |\log(\al)|$. The following proposition provides more precise information about this upper bound, pushing maximisers below the value \(B\). While the idea of the proposition is simple, the proof requires quite a rigorous balancing of terms and estimates. This is to us one of the first quantitative size estimates for a 'small amplitude' theory (which establish that they are not small at all in fact), and it takes up a sizeable part of the paper. We shall use it in the following, but it is strictly not needed for the construction of very small waves. The reader who is dominantly interested in the construction method may move forward to Lemma~\ref{lemma:inner product}.

\begin{proposition}
\label{prop:upper bound on f_l}
For any $\al>0$ and $\varepsilon>0$, there is an $l_0$, depending only on $\varepsilon$, such that for all $l>l_0$, the maximizers $f_l$ satisfy
\begin{equation}
	\int_{f_l\geq B}|f_l-B|\dx<\varepsilon,
\end{equation}
where $B=\frac{4}{\sqrt{3}}\cos\left(\frac{5\pi}{18}\right)\alpha \approx 1.48 \alpha$ is the unique solution in $\rn_+$ to $2\Psi(y)=y\Psi'(y)$.
\end{proposition} 

\begin{proof}
Let $\varepsilon>0$. For a contradiction, we assume that for any $l_0$, there exists maximizers $f_l$, $l>l_0$, such that
\begin{equation}\label{eq:assumption f_l}
\int_{f_l\geq B}|f_l-B|\dx\geq \varepsilon.
\end{equation}
The idea of the proof is that $y^2/\Psi(y)$ has its maximum at $B$ so that $B$ is the optimal height for maximizing the $L^2$-norm for fixed $N_\Psi$-norm, as shown below. Hence we can construct a function $\tilde{f_l}$ with $\supp(\tilde{f_l}) \subset [-2^l,2^l]$, $\tilde{f_l}(0)=B$ and $N_\Psi(\tilde{f_l})=1$ that is "flatter" than $f_l$ and such that $\|\tilde{f_l}\|_{L^2}>\|f_l\|_{L^2}$. Generally, the difference between $\|K_\frac{1}{4}\ast f\|_{L^2}$ and $\|f\|_{L^2}$ is smaller the "flatter" and less oscillating $f$ is (see \eqref{eq:slobodeckij}). The restriction to $[-2^l,2^l]$ causes some technical difficulties due to the jump-discontinuity at the endpoints, but for large enough $l$ we can show that $\|K_\frac{1}{4}\ast \tilde{f_l}\|_{L^2}>\|K_\frac{1}{4}\ast f_l\|_{L^2}$, contradicting the assumption that $f_l$ is a maximizer.

Let now $l>0$ be such that 
\begin{equation}
\label{eq:lower bound on l}
\int_{-2^l}^{2^l}\Psi(B)\dx>1.
\end{equation}
For $y>0$, we define
\begin{equation*}
	g(x,y) =
	\begin{cases}
	y, \quad & |x|\leq \frac{1}{2\Psi(y)} \\
	0, \quad & |x|> \frac{1}{2\Psi(y)}.
	\end{cases}
\end{equation*}
Then $N_\Psi(g(\cdot,y))=1$ for all $y>0$. Let
\begin{equation*}
	h(y)=\|g(\cdot,y)\|_{L^2}^2=\frac{y^2}{\Psi(y)}.
\end{equation*}
We then have that
\begin{equation}
	\label{eq:derivative of h}
	h'(y)=\frac{2y}{\Psi(y)}-y^2\frac{\Psi'(y)}{\Psi(y)^2}=\frac{y}{\Psi(y)^2}(2\Psi(y)-y\Psi'(y)).
\end{equation}
Using the definition of $\Psi$, some straightforward calculations show that $2\Psi(y)-y\Psi'(y)=0$ (a cubic equation) has a single real-valued solution, namely $B$, whence $h'(B)=0$. Moreover,
\begin{align}\label{eq:h'}
	(B-y)h'(y)>0 & \quad \text{ for } \quad 0 < y \neq B.  
\end{align}
Let $(-a,a)$ be the interval where $f_l>B$. For $\delta>0$, we define
\begin{equation*}
	\tilde{f}(x)=
	\begin{cases}
	B, \quad & |x| \leq a+\delta, \\
	f_l(|x|-\delta), \quad & |x|>a+\delta. 
	\end{cases}
\end{equation*}
In view of the assumption \eqref{eq:assumption f_l} and \(N_\Psi(f_l)=1\), one has \(N_\Psi(\tilde f) <1\) for \(\delta = 0\). As the gauge norm is continuous and grows unboundedly as \(\delta \to \infty\), there exists a smallest $\delta>0$ such that $N_\Psi(\tilde{f})=1$. Fix that $\delta>0$. It is obvious that
\begin{equation*}
\int\limits_{|x| > a+\delta} \tilde{f}^2 \dx=\int\limits_{|x| > a} f_l^2 \dx, \quad\text{ and }\quad	\int\limits_{|x| > a+\delta} \Psi(\tilde{f})\dx=\int\limits_{|x| > a} \Psi(f_l)\dx,
\end{equation*}
and by choice of \(\delta\), also
\begin{equation}\label{eq:equal Psi}
	\int\limits_{|x| \leq a+\delta} \Psi(\tilde{f})\dx=\int\limits_{|x| \leq a} \Psi(f_l)\dx.
\end{equation}
Now, in the whole interval \((-a-\delta, a + \delta)\), one has
\[
h(\tilde f) = h(B) = \max h \geq h(f_l),
\]
in view of \eqref{eq:h'}; and in $(-a,a)$, where has $f_l(x)>B=\tilde{f}(x)$, one has
\[
h(\tilde f) = h(B) = \max h > h(f_l).
\]
This means that \(y^2\) is relatively smaller than \(\Psi(y)\) in the range where \(f_l\) resides. More precisely,
\begin{align*}
	\|\tilde{f}\|_{L^2}^2-\|f_l\|_{L^2}^2 &= \int\limits_{|x| \leq a+ \delta} \tilde{f}^2\dx-\int\limits_{|x| \leq a} f_l^2\dx\\
	&= \int\limits_{|x| \leq a+ \delta} \Psi(\tilde f) h(\tilde{f}) \dx-\int\limits_{|x| \leq a} \Psi(f_l) h(f_l) \dx\\
	&= h(B) \Bigg( \int\limits_{|x| \leq a+ \delta} \Psi(\tilde f)  \dx - \int\limits_{|x| \leq a} \Psi(f_l) \frac{h(f_l)}{h(\tilde f)} \dx \Bigg) > 0,
\end{align*}
in view of \eqref{eq:equal Psi} and the above inequalities for \(h\). We shall quantify this inequality. Using \eqref{eq:equal Psi} and the definition of $\tilde f$, one finds the following explicit relationship between \(f_l\) and \(\tilde f\): 
\begin{equation*}
\int\limits_{|x| \leq a+\delta} \tilde{f}^2\dx=2(a+\delta)B^2= \frac{B^2}{\Psi(B)}\int\limits_{|x| \leq a + \delta} \Psi(\tilde f)\dx = \frac{B^2}{\Psi(B)}\int\limits_{|x| \leq a} \Psi(f_l)\dx.
\end{equation*}
Hence, by the Taylor expansion
\[
\Psi(f_l) = \Psi(B) + \Psi'(B) (f_l-B) + {\textstyle \frac{1}{2}} \Psi''(c(x)) (f_l-B)^2,
\]
for some \(c(x)\in (B, f_l(x))\), one finds
\begin{equation}\label{eq:L2-Taylor}
\begin{aligned}
\|\tilde{f}\|_{L^2}^2-\|f_l\|_{L^2}^2= & \frac{1}{\Psi(B)}\int_{-a}^a B^2 \Psi(f_l)-\Psi(B)(B+(f_l-B))^2\dx \\
= & \frac{1}{\Psi(B)}\int_{-a}^a \left(B^2 \Psi'(B)-2\Psi(B)B\right)(f_l-B)\dx \\
&+\frac{1}{\Psi(B)}\int_{-a}^a\left(\frac{B^2}{2}  \Psi''(c(x))-\Psi(B)\right) (f_l-B)^2\dx.
\end{aligned}
\end{equation}
From \eqref{eq:derivative of h}, we have $B^2 \Psi'(B)-2\Psi(B)B=0$ as $h'(B)=0$. And since $\Psi''$ is strictly increasing on $[\alpha,\infty)$, we get
\begin{equation*}
\frac{B^2}{2}\Psi''(c(x))-\Psi(B)>\frac{B^2}{2}\Psi''(B)-\Psi(B) \geq \alpha^3,
\end{equation*}
by an explicit calculation.
It follows from \eqref{eq:assumption f_l} and Jensen's inequality that 
\begin{equation}
\label{eq:L2 norm difference tilde f and f_l}
\|\tilde{f}\|_{L^2}^2-\|f_l\|_{L^2}^2 =\frac{1}{\Psi(B)}\int_{-a}^a\left(\frac{B^2}{2}\Psi''(c(x))-\Psi(B)\right) (f_l-B)^2\dx \gtrsim \varepsilon^2,
\end{equation}
uniformly for all $l$ sufficiently large (note also that there is a uniform upper bound on $a$, imposed by $N_\Psi(f_l)=1$). Using the same kind of Taylor expansion, we can also get the following row of equalities, that will be used later. Here, the first equality is a rewrite of \eqref{eq:equal Psi}, and the last of \eqref{eq:L2-Taylor}. 
\begin{align}
2\delta = & \frac{1}{\Psi(B)}\int_{-a}^a \Psi(f_l)-\Psi(B)\dx=\int_{-a}^a \frac{\Psi'(B)}{\Psi(B)}(f_l-B)+\frac{\Psi''(c(x))}{2\Psi(B)}(f_l-B)^2\dx \nonumber \\
= & \frac{2}{B}\int_{-a}^a (f_l-B)\dx+\frac{1}{B^2}(\|\tilde{f}\|_{L^2}^2-\|f_l\|_{L^2}^2)+\frac{1}{B^2}\int_{-a}^a (f_l-B)^2\dx. \label{eq:expression for delta}
\end{align}
Next we want to show that $\|K_{\frac{1}{4}}\ast \tilde{f}\|_{L^2}>\|K_{\frac{1}{4}}\ast f_l\|_{L^2}$ if $l$ is sufficiently large. As $\|K\|_{L^1}=1$, we have for any $f\in L^2$ that
\begin{align}
	\iint_{\rn^2} |f(x+h)-f(x)|^2K(h)\dx\, \mathrm{d}h = &\iint_{\rn^2} |\widehat{f}(\xi)|^2 \big|\mathrm{e}^{ih\xi}-1\big|^2 K(h)\dxi \, \mathrm{d}h \nonumber\\
	= & \int_{\rn} |\widehat{f}(\xi)|^2 \int_{\rn} 2(1-\cos(h\xi))K(h) \, \mathrm{d}h\dxi \nonumber\\
	= & 2\int_{\rn} (1-\widehat{K}(\xi))|\widehat{f}(\xi)|^2\dxi \nonumber\\
	= & 2\left( \|f\|_{L^2}^2-\|K_\frac{1}{4}\ast f\|_{L^2}^2\right). \label{eq:slobodeckij}
\end{align}
Hence,
\begin{equation}
\begin{aligned}\label{eq:difference Kf and K tilde f} 
\|K_{\frac{1}{4}}\ast \tilde{f}\|_{L^2}^2 &-\|K_{\frac{1}{4}}\ast f_l\|_{L^2}^2 =  \|\tilde{f}\|_{L^2}^2-\|f_l\|_{L^2}^2 \\
& -\frac{1}{2} \iint_{\rn^2} \left(|\tilde{f}(x+h)-\tilde{f}(x)|^2-|f_l(x+h)-f_l(x)|^2\right)K(h)\dx\, \mathrm{d}h. 
\end{aligned}
\end{equation}
Since $\|\tilde{f}\|_{L^2} - \|f_l\|_{L^2} \gtrsim \varepsilon^2$, we only need to show that the double integral in the second line in \eqref{eq:difference Kf and K tilde f} is smaller than that expression.
By evenness of \(\tilde f\), \(f_l\) and \(K\), this integral is equal to
\begin{equation}\label{eq:the reduced integral}
\int_0^\infty \left( \int_\rn \left(|\tilde{f}(x+h)-\tilde{f}(x)|^2-|f_l(x+h)-f_l(x)|^2\right) \dx \right) K(h)\, \mathrm{d}h,
\end{equation}
where we shall concentrate on the inner integral. Hence, let \(h \geq 0\). For \(|x| \geq a + \delta\), the function \(\tilde f\) is just a \(\delta\)-translation of \(f_l\), so 
\begin{align*}
&\int_\rn | \tilde f(x+h) - \tilde f(x)|^2  \dx\\ 
&= \left( \int_{-\infty}^{-a-\delta-h}  +  \int_{-a-\delta-h}^{a+\delta} +  \int_{a+\delta}^\infty\right)  | \tilde f(x + h) - \tilde f(x)|^2 \dx\\
&=  \left( \int_{-\infty}^{-a-h} + \int_a^\infty \right)  | f_l(x + h) - f_l(x)|^2 \dx+  \int_{-a-\delta-h}^{a+\delta}  | \tilde f(x + h) - \tilde f(x)|^2 \dx 
\end{align*}
Hence, the inner integral in \eqref{eq:the reduced integral} reduces to
\begin{equation}\label{eq:E(h)}
\begin{aligned}
E_l(h) &:= \int_\rn \left( | \tilde f(x+h) - \tilde f(x) |^2 - | f_l(x+h) - \tilde f_l(x) |^2 \right) \dx\\ 
&\;= \int_{-a-\delta-h}^{a+\delta}  | \tilde f(x + h) - \tilde f(x)|^2 \dx -  \int_{-a-h}^a  | f_l(x + h) - f_l(x)|^2 \dx. 
\end{aligned}
\end{equation}
We study \(0 \leq h \leq 2a\) and \(h \geq 2(a+\delta)\) separately. In the following, we will make ample use of the definition of \(\tilde f\), as well as translations and (even) changes of variables in the integrals to reduce and compare the terms. Recall that \(\tilde f = B\) on \([-a-\delta, a + \delta]\).
 
\subsection*{The case \(0 \leq h \leq 2a\).} In the simplest case, when \(h \in [0, 2a]\), one has
\begin{align*}
\int_{-a-\delta-h}^{a+\delta} | \tilde f(x+h)-\tilde f(x)|^2 \dx &= \left( \int_{-a-\delta-h}^{-a-\delta} + \int_{-a-\delta}^{a+\delta-h} + \int_{a+\delta-h}^{a+\delta} \right) | \tilde f(x+h)-\tilde f(x)|^2 \dx\\
&= \int_{-a-h}^{-a} |B-f_l(x)|^2 \dx + \int_{a-h}^a |f_l(x+h)-B|^2 \dx\\
&= 2\int_{-a-h}^{-a} |B-f_l(x)|^2 \dx,
\end{align*}
by the changes of variables \(x + h \mapsto x \mapsto -x\). So the correction term \(E_l(h)\) from \eqref{eq:E(h)} in the case when \(h \in [0,2a]\) is given by
\begin{align*}
&2\int_{-a-h}^{-a} |B-f_l(x)|^2 \dx - \int_{-a-h}^a  | f_l(x + h) - f_l(x)|^2 \dx \\
&= 2\int_{-a-h}^{-a} |B-f_l(x)|^2 \dx - \left( \int_{-a-h}^{-a} + \int_{-a}^{a-h} - \int_{a-h}^{a}\right)  | f_l(x + h) - f_l(x)|^2 \dx \\
&= 2\int_{-a-h}^{-a} \left( |B-f_l(x)|^2 - | f_l(x + h) - f_l(x)|^2 \right) \dx -  \int_{-a}^{a-h} | f_l(x + h) - f_l(x)|^2 \dx.
\end{align*}
Both latter terms are negative, as \(f_l(\cdot+h) \geq B\) on \([-a-h,-a]\) when \(h \leq 2a\), and \(f_l(\cdot) \leq B\) on the same interval when \(h \geq 0\). To quantify the negative contribution, we relate it to the norm difference \eqref{eq:L2 norm difference tilde f and f_l} via the expression \(\int_{-a}^a |B- f_l|^2 \dx\). Since 
\[
|B-f_l(x)|^2 - | f_l(x + h) - f_l(x)|^2 = B^2 - 2Bf_l(x) - f_l^2(x+h) + 2f_l(x) f_l(x+h),
\]
one gets after the change of variables \(x + h \mapsto x \mapsto -x\), and subsequent addition and subtraction of \(|B-f_l(x)|^2\), that
\begin{align*}
&2\int_{-a-h}^{-a} \left( |B-f_l(x)|^2 - | f_l(x + h) - f_l(x)|^2 \right) \dx\\
&= -2\int_{-a}^{a-h} \left( |B-f_l(x)|^2 + 2(f_l(x)-B)(B-f_l(x+h)) \right) \dx\\
&\leq -2\int_{-a}^{a-h} |B-f_l(x)|^2  \dx.
\end{align*}
Therefore, going back to \eqref{eq:the reduced integral}, when \(h \in [0,2a]\), the correction to the \(L^2\)-terms in \eqref{eq:difference Kf and K tilde f} can be bounded as
\begin{equation}\label{eq:E(h), h<2a}
\begin{aligned}
\int_0^{2a} E_l(h) K(h)\, \mathrm{d}h &\leq -2\int_0^{2a} \int_{-a}^{a-h} | B - f_l(x) |^2 \dx \, K(h) \, \mathrm{d}h\\
&\leq -2\int_0^{a} \int_{-a}^{0} | B - f_l(x) |^2 \dx \, K(h) \, \mathrm{d}h\\
&\simeq -a\left(\|\tilde{f}\|_{L^2}^2 -\|f_l\|_{L^2}^2\right),
\end{aligned}
\end{equation}
by \eqref{eq:L2 norm difference tilde f and f_l}. Moreover, by \eqref{eq:assumption f_l} and \eqref{eq:bounds on f_l}, there is a lower bound on $a$ that depends only on $\varepsilon$ and $\alpha$; in particular it is uniform in $l$.

\subsection*{The case \(h \geq 2a\).} We divide the integral \( \int_{-a-\delta-h}^{a+\delta}  | \tilde f(x + h) - \tilde f(x)|^2 \dx\) in \eqref{eq:E(h)} according to
\begin{equation}\label{eq:five integrals}
\int_{-a-\delta-h}^{a + \delta} = \int_{-a-\delta-h}^{a-\delta - h} + \int_{a-\delta-h}^{a+\delta - h} + \int_{a+\delta-h}^{-a-\delta} + \int_{-a-\delta}^{-a+\delta} + \int_{-a + \delta}^{a + \delta}.   \tag{I}
\end{equation}
Note that \(a + \delta - h = -a - \delta\) exactly when \(h = 2(\alpha + \delta)\), so the behaviour of the middle integrals in the right-hand side will depend on whether \(2a \leq h \leq 2(a + \delta)\) or \(h \geq 2(a + \delta)\). We deal with the integrals in order. When \(x \in (-a-\delta-h,a-\delta - h)\), \(\tilde f(x+h)\) is constantly equal to \(B\), and \(\tilde f(x)\) is a \(\delta\)-translation of \(f_l(x)\), so one has
\begin{equation}\label{eq:I1}
\int_{-a-\delta-h}^{a-\delta - h}  | \tilde f(x + h) - \tilde f(x)|^2 \dx = \int_{-a - h}^{a - h} |B- f_l(x)|^2 \dx \tag{I1}
\end{equation}
Using the changes of variables \(x + h \mapsto x \mapsto -x\), the second integral in \eqref{eq:five integrals} may be rewritten as a copy of the fourth:
\begin{align}\label{eq:I2}
\int_{a-\delta-h}^{a+\delta - h}  | \tilde f(x + h) - \tilde f(x)|^2 \dx &= \int_{-a-\delta}^{-a+\delta}  | \tilde f(x) - \tilde f(x+h)|^2 \dx. \tag{I2}
\end{align}
The third integral is only present, or only contributing positively, when \(h \geq 2(a + \delta)\). In that case both \(x + h \geq 2 + \delta\) and \(x \leq -a - \delta\) on the interval, so we count its contribution as
\begin{equation}\label{eq:I3}
\begin{aligned}
&\chi_{(2(a+\delta),\infty]}(h) \int_{a+\delta-h}^{-a-\delta}|f_l(x+h - \delta)-f_l(x+\delta)|^2 \dx\\ 
 &= \chi_{(2(a+\delta),\infty]}(h) \int_{a-h}^{-a-2\delta}|f_l(x+h)-f_l(x+2\delta)|^2 \dx,  
\end{aligned}
\tag{I3}
\end{equation}
where \(\chi\) is an indicator function. The fourth integral has already been show to be a copy of the second, see \eqref{eq:I2}, but its behaviour is linked to the size of \(h\) in relation to \(2(a+\delta)\). When the interval \([-a-\delta+h,a+\delta]\) is non-void, both \(\tilde f(x+h)\) and \(\tilde f(x)\) equal \(B\) there, so the integral vanishes over that part, else only one of these terms is constant. Therefore,
\begin{equation}\label{eq:I4}
\begin{aligned}
\eqref{eq:I2} = \int_{-a-\delta +h}^{-a +\delta + h} |\tilde f(x) - B|^2 \dx &= \int_{\max\{a+\delta, -a-\delta +h\}}^{-a +\delta + h} |f_l (x-\delta) - B|^2 \dx\\
&= \int_{a - h}^{\min\{-a, a+2\delta -h\}} |f_l (x) - B|^2 \dx, 
\end{aligned}
\tag{I4}
\end{equation}
by the change of variables \(x - \delta \mapsto x \mapsto -x\). Finally, the fifth can be shown to be a copy of the first:
\begin{equation}\label{eq:I5}
\begin{aligned}
\int_{-a + \delta}^{a + \delta}  | \tilde f(x + h) - \tilde f(x)|^2 \dx &= \int_{-a + \delta}^{a + \delta}  |  f_l(x + h - \delta) - B|^2 \dx\\ &= \int_{-a-h}^{a -h}  |  f_l(x) - B|^2 \dx = \eqref{eq:I1},
\end{aligned}
\tag{I5}
\end{equation}
by the change of variables \(x + h - \delta \mapsto x \mapsto - x\). Let \(r(h) = \min\lbrace a+2\delta-h,-a\rbrace\). Then \cref{eq:I1,eq:I2,eq:I3,eq:I4,eq:I5} summarise as 
\begin{equation} \label{eq:tilde f x+h - tilde f x} 
\begin{aligned}
\int_{-a-\delta-h}^{a+\delta}|\tilde{f}(x+h)-\tilde{f}(x)|^2\dx=&2\int_{-a-h}^{r(h)}|f_l(x)-B|^2\dx \\
&+\chi_{(2(a+\delta),\infty]}(h)\int_{a-h}^{-a-2\delta}|f_l(x+h)-f_l(x+2\delta)|^2\dx. 
\end{aligned}
\end{equation}
This shall be compared with the term \(\int_{-a-h}^a  | f_l(x + h) - f_l(x)|^2 \dx\) from \eqref{eq:E(h)} in \(E_l(h)\). We divide the latter term into parts, according to
\begin{equation}\label{eq:the f_l part}
\begin{aligned}
-  \int_{-a-h}^a  | f_l(x + h) - f_l(x)|^2 \dx &= -  \left(\int_{-a-h}^{a-h}  +  \int_{a-h}^{-a} + \int_{-a}^a \right)  | f_l(x + h) - f_l(x)|^2 \dx\\
&=   - \left(2 \int_{-a-h}^{a-h}  +  \int_{a-h}^{-a} \right)  | f_l(x + h) - f_l(x)|^2 \dx,
\end{aligned} 
\end{equation}
by the changes of variables. The \(2\int_{-a-h}^{a-h}\)-integral will be used for the \(2\int_{-a-h}^{r(h)}\)-integral in \eqref{eq:tilde f x+h - tilde f x}, and the \(\int_{a-h}^{-a}\)-integral for the \(\chi\)-part. We start with the former, writing
\begin{equation}\label{eq:r(h) part}
\begin{aligned}
&2\int_{-a-h}^{r(h)}|f_l(x)-B|^2\dx-2\int_{-a-h}^{a-h} |f_l(x+h)-f_l(x)|^2\dx\\
&= 2\int_{a-h}^{r(h)}|f_l(x)-B|^2\dx + 2\int_{-a-h}^{a-h}|f_l(x)-B|^2\dx-2\int_{-a-h}^{a-h} |f_l(x+h)-f_l(x)|^2\dx\\ 
&= 2\int_{a-h}^{r(h)}|B-f_l(x)|^2\dx\\
&\quad -4\int_{-a-h}^{a-h} (B-f_l(x))(f_l(x+h)-B)\dx  -2\int_{-a}^{a}|f_l(x)-B|^2\dx,
\end{aligned}
\end{equation}
which is achieved by adding and subtracting \(2 \int_{-a-h}^{a-h} |f_l(x+h) - B|^2 \dx\), where the negative term is expressed as \(-2 \int_{-a}^{a} |f_l(x) - B|^2 \dx\). Note that the terms \({B-f_l(x)}\) and \(f_l(x+h)-B\) in the \(\int_{-a-h}^{a-h}\)-integral are both positive, so the total contribution from that term is negative. Quantifying with the help of the mean value theorem, we have that \eqref{eq:r(h) part} equals
\begin{equation}\label{eq:r(h) part II}
\begin{aligned}
&2\int_{a-h}^{r(h)}|B-f_l(x)|^2\dx-4\int_{-a-h}^{a-h} (B-f_l(x))(f_l(x+h)-B)\dx  -2\int_{-a}^{a}|f_l(x)-B|^2\dx\\
&= 2(r(h)-a+h)(B-f_l(c_1))^2_{c_1 \in (a-h,-a)}\\ 
&\quad -4(B-f_l(c_2))_{c_2 \in (-a-h,a-h)}\int_{-a}^{a} |f_l(x)-B|\dx  - 2\int_{-a}^{a}|f_l(x)-B|^2\dx,
\end{aligned}
\end{equation}
where furthermore \(f(c_2) \leq f(c_1) \leq B\) and \(0 \leq r(h) - a + h \leq 2\delta\) for \(h \geq 2a\) by definition of \(r(h)\). We may therefore use \eqref{eq:expression for delta} to bound \eqref{eq:r(h) part II}, finding that it is less than
\begin{equation}\label{eq:E(h) for h>2a}
\begin{aligned}
& 4|B-f_l(c_1)|\left(\delta (B-f_l(c_1)) -\int_{-a}^a|f_l(x)-B|\dx\right)-2\int_{-a}^{a}|f_l(x)-B|^2\dx \\
= & 2\left(1-\frac{f_l(c_1)}{B}\right)\left(\|\tilde{f}\|_{L^2}^2  -\|f_l\|_{L^2}^2\right) - \frac{2 f_l(c_1)}{B}\int_{-a}^a | f_l(x) - B|^2 \dx -4\delta(B-f_l(c_1))f_l(c_1)\\
&\leq 2 \left(\|\tilde{f}\|_{L^2}^2  -\|f_l\|_{L^2}^2\right).  
\end{aligned}
\end{equation}
Turning to the \(\chi\)-part of \eqref{eq:tilde f x+h - tilde f x}, it is only present when $h>2(a+\delta)$, in which case we are to balance it against the remaining \(\int_{a-h}^{-a}\)-integral in \eqref{eq:the f_l part}. Note that changes of variables can be used to re-express terms, for example, \(\int_{a-h}^{-a-2\delta} f_l^2(x+h) \dx = \int_{a-h}^{-a-2\delta} f_l^2(x+2\delta) \dx\) by the change \(x+h \mapsto -x + 2\delta\). Using similar identities and the mean-value theorem, one finds 
\begin{align}
\int_{a-h}^{-a-2\delta} & |f_l(x+h)-f_l(x+2\delta)|^2\dx-\int_{a-h}^{-a} |f_l(x+h)-f_l(x)|^2\dx \nonumber\\
= & -2\int_{a-h}^{-a-2\delta}f_l(x+h)(f_l(x+2\delta)-f_l(x))\dx \nonumber\\
&+2\int_{-a-2\delta}^{-a}f_l(x+h)(f_l(x)-f_l(x+h))\dx \nonumber\\
= & -2f_l(c_3)\int_{a-h}^{-a-2\delta} ( f_l(x+2\delta)-f_l(x) ) \dx+2f_l(c_4)\int_{-a-2\delta}^{-a} ( f_l(x)-f_l(x+h) ) \dx \nonumber\\
= &-2(f_l(c_3)-f_l(c_4))\left(\int_{-a-2\delta}^{-a} f_l(x)\dx-\int_{-a-2\delta+h}^{-a+h}f_l(x)\dx\right)<0, \label{eq:remaining terms}
\end{align}
where $c_3\in (a,-a-2\delta+h)$, $c_4\in (-a-2\delta+h,-a+h)$. As $f_l$ is bell-shaped, it follows that the expression above is negative. As $K$ is positive and $\int_0^\infty K(h) \, \mathrm{d}h =\frac{1}{2}$, summing up \(\int_0^{\infty} E_l(h) K(h)\, \mathrm{d}h\) from the negative  quantified contribution from \eqref{eq:E(h), h<2a}, the  positive quantified from \eqref{eq:E(h) for h>2a}, and the negative from \eqref{eq:remaining terms}, we find that there is a constant $0<C<1$, depending only on $\varepsilon$ and $\alpha$ such that
\begin{align*}
\frac{1}{2} \iint_{\rone^2} \left(|\tilde{f}(x+h)-\tilde{f}(x)|^2-|f_l(x+h)-f_l(x)|^2\right)K(h)\dx\, \mathrm{d}h <C \left(\|\tilde{f}\|_{L^2}^2-\|f_l\|_{L^2}^2\right),
\end{align*}
for all $l$ sufficiently large. From \eqref{eq:difference Kf and K tilde f}, we then get that
\begin{equation}
\label{eq: K tilde f is bigger than K f_l}
\|\Kf\ast \tilde{f}\|_{L^2}^2 -\|\Kf\ast f_l\|_{L^2}^2>(1-C) \left(\|\tilde{f}\|_{L^2}^2-\|f_l\|_{L^2}^2\right)>0.
\end{equation}

\subsection*{The modified \(\tilde f\)}
Because $\tilde{f}$ has support in $[-2^l-\delta,2^l+\delta]$ it is not an admissible maximizer, and we now modify it to yield the desired contradiction. By \eqref{eq:lower bound on l} and the properties of \(\tilde f\), there exists $\gamma>0$ such that $a+\delta+\gamma<2^l$ and
\begin{equation*}
\int_0^{\gamma}\Psi\left(B\right) \dx=\gamma\Psi(B)=\int_{-2^l}^{-2^l+\delta +\gamma}\Psi(f_l)\dx.
\end{equation*}
Now set $\tilde{\delta}=\delta+\gamma$, and define $\tilde{f_l}$ by
\begin{equation*}
\tilde{f_l}(x)=
\begin{cases} 
B, \quad & x\in [-a-\tilde{\delta},a+\tilde{\delta}], \\
f_l(x-\tilde{\delta}), \quad & x\in(a+\tilde{\delta},2^l), \\
f_l(x+\tilde{\delta}), \quad & x\in (-2^l,-a-\tilde{\delta}), \\
0, \quad & |x|\geq 2^l. \\
\end{cases}
\end{equation*}
Then $\supp(\tilde{f_l}) \subset [-2^l,2^l]$ and $N_\Psi(\tilde{f}_l)=1$. We claim that
\begin{equation}
\label{eq:difference tilde f and tilde f_l}
\lim_{l\rightarrow \infty} \left(\|\tilde{f}-\tilde{f_l}\|_{L^2}+\|\Kf\ast \tilde{f}-\Kf\ast \tilde{f_l}\|_{L^2}\right)=0,
\end{equation}
in which case it follows from \eqref{eq: K tilde f is bigger than K f_l} and \eqref{eq:L2 norm difference tilde f and f_l} that
\begin{equation*}
\|\Kf\ast \tilde{f_l}\|_{L^2}^2 -\|\Kf\ast f_l\|_{L^2}^2>0
\end{equation*}
for all $l$ sufficiently large, contradicting the assumption that $f_l$ is a maximizer. To prove \eqref{eq:difference tilde f and tilde f_l}, note that
\begin{equation}\label{eq:diff tilde f and tilde f_l}
\begin{aligned}
	&\int  |\tilde{f}(x)-\tilde{f}_l(x)|^2\dx \\
	&= 2\left(\int_{-2^l-\delta}^{-2^l}\tilde{f}^2\dx+\int_{-2^l}^{-a-\tilde{\delta}}|\tilde{f}(x)-\tilde{f}_l(x)|^2\dx+\int_{-a-\tilde{\delta}}^{-a-\delta}|B-\tilde{f}(x)|^2\dx\right) \\
	& \leq  2\delta f_l(-2^l+\delta)^2+2B^2 \gamma +2\int_{-2^l+\delta}^{-a-\gamma}|f_l(x+\gamma)-f_l(x)|^2\dx. 
\end{aligned}
\end{equation}
Let $f$ be a non-negative, bell-shaped function satisfying $N_\Psi (f)=1$. As $f$ is bell-shaped, we have that if $f(x)=k$, then $f(y)\geq k$ and $\Psi(f(y))\geq \Psi(k)$ for all $|y|\leq |x|$. The condition $N_\Psi(f)=1$ thus implies that
\begin{equation}
\label{eq:upper bound on f}
f(x)\leq \Psi^{-1} \left(\frac{1}{2|x|}\right), 
\end{equation}
for all $x\in \rone$, and that for all $k>0$, $|\lbrace x : f(x)\geq k\rbrace|\leq \frac{1}{\Psi(k)}$. In particular, there is an upper bound on $\tilde{\delta}$ that is independent of $l$ and $\varepsilon$, say $\tilde{\delta}<C$, and 
\begin{equation*}
f_l(-2^l+\tilde{\delta})\leq \Psi^{-1} \left(\frac{1}{2|-2^l+C|}\right)\rightarrow 0
\end{equation*}
as $l\rightarrow \infty$, uniformly in $\varepsilon$. This also implies that $\gamma\rightarrow 0$. Hence the two first terms in \eqref{eq:diff tilde f and tilde f_l} vanish and, to prove our claim, it is sufficient to show that the convergence
\begin{equation*}
	\lim_{h\rightarrow 0^+}\int_{|x|<2^l-h}|f_l(x+h)-f_l(x)|^2\dx =0
\end{equation*}
is uniform in $l$.
Let $h>0$ and $x\in(-2^l,2^l-h)$. As $J_l$ is bounded below and $\langle f_l, \Psi'(f_l)\rangle$ is bounded above, uniformly in $l$, we get from the Euler-Lagrange equation \eqref{eq:local Euler} that
\begin{align*}
	\left|\Psi'(f_l(x+h))-\Psi'(f_l(x))\right|= &\frac{\langle f_l, \Psi'(f_l)\rangle }{J_l^2}\left|K\ast f_l(x+h)-K\ast f_l(x)\right|  \\
	\lesssim & \|f_l\|_{L^\infty}\int \left| K(x+h-y)-K(x-y)\right|\dy \\
	\lesssim & \|f_l\|_{L^\infty} h^\frac{1}{2},
\end{align*}
uniformly in $l$, where we used the regularity and decay properties of $K$ (cf. Lemma \ref{lemma:K_alpha}). As $\|f_l\|_{L^\infty}\lesssim \al$ uniformly in $l$, $\Psi'(x)>0$ for all $x>0$ and $\Psi''(x)>0$ for all $x\neq \al$, this implies that $f_l(x+h)-f_l(x)\rightarrow 0$ as $h\rightarrow 0$ uniformly in $l$ and $x$. In particular, for any fixed $R$,
\begin{equation*}
	\int_{|x|\leq R} |f_l(x+h)-f_l(x)|^2\dx \rightarrow 0,
\end{equation*}
uniformly in $l$.
To deal with $|x|>R$, we improve on the estimate of the difference. By \eqref{eq:upper bound on f} we can pick $R\gg 1$ such that $f_l(R)<\frac{\al}{2}$ for all $l$. Let $|x|>R\gg 1$ and $0<h\ll 1$. Then
\begin{align*}
	\left|\Psi'(f_l(x+h))-\Psi'(f_l(x))\right| & =|f_l(x+h)-f_l(x)|(2\al-f_l(x+h)-f_l(x))\\
	&\geq \al |f_l(x+h)-f_l(x)|,
\end{align*}
and, by Lemma \ref{lemma:K_alpha} and the bell-shapedness of $f_l$,
\begin{align*}
	|K\ast &f_l(x+h)-K\ast f_l(x)|= \left|\int \left(K(x+h-y)-K(x-y)\right)f_l(y)\dy\right| \\
	&\simeq  \int_{|y|<\frac{|x|}{2}}h|K'(x-y)|f(y)\dy +\int_{|y|\geq \frac{|x|}{2}}\left|K(x+h-y)-K(x-y)\right| f(y)\dy \\
	&\leq  h|x|\left|K'\left(\frac{|x|}{2}\right)\right|\|f_l\|_{L^\infty}+f_l\left(\frac{x}{2}\right)\int \left|K(x+h-y)-K(x-y)\right|\dy \\
	&\lesssim_N  h \al |x|^{-N}+h^\frac{1}{2}f_l\left(\frac{x}{2}\right),
\end{align*}
uniformly in $l$ for any $N>0$. Picking any $N>1$, it follows from the Euler-Lagrange equation \eqref{eq:local Euler} and the above estimates that
\begin{align*}
	\int_{R<|x|<2^l-h}|f_l(x+h)-f_l(x)|^2\dx \lesssim & \int_{|x|>R} |K\ast f_l(x+h)-K\ast f_l(x)|^2 \dx \\
	\lesssim & h^2 \al^2+h \|f_l\|_{L^2}^2 \\
	\lesssim & h^2\al^2+\frac{h}{\al},
\end{align*}
uniformly in $l$ and $\al$. This proves the claim. As $\varepsilon>0$ was arbitrary, this proves the result.
\end{proof}

\begin{lemma}\label{lemma:inner product}
The quantity \(\langle f, \Psi'(f) \rangle\) is bounded from below by \(N_\Psi(f)\) for any \(f \in \mathcal{L}^{\Psi}\). For the maximisers \(f_l\), one has \(1 < \langle f_l, \Psi'(f_l) \rangle\) and
\begin{align*}
\langle f_l, \Psi'(f_l) \rangle = & 2 -   \frac{1}{3}  \int_{f_l < \alpha}    f_l^3\dx+  \int_{f_l \geq  \alpha} \alpha^2  \left( {\textstyle\frac{2\alpha}{3}} - f_l\right)  + \left( (f_l-\alpha)^3 + 3\alpha(f_l-\alpha)^2 \right)\dx \\
\lesssim & 1
\end{align*}
uniformly in \(l \in \N\) and \(\alpha \in (0,\infty)\). Moreover, for any $\al>0$, there exists $l_0$ such that
\begin{equation*}
	\langle f_l, \Psi'(f_l) \rangle <2,
\end{equation*}
for all $l>l_0$.
\end{lemma}

\begin{proof}
For \(f <\alpha\) one has
\begin{align*}
f \Psi'(f) &= 2 \alpha f^2 - f^3 = 2\left( \alpha f^2 - {\textstyle \frac{1}{3}} f^3 \right) -  {\textstyle \frac{1}{3}} f^3 = 2 \Psi(f) -  {\textstyle \frac{1}{3}} f^3>\Psi(f),
\end{align*}
and, for \(f \geq \alpha\),
\begin{align*}
f \Psi'(f) &= \alpha^2 f  + 3(f-\alpha)^2 f\\
&= 2 \left( {\textstyle \frac{2}{3}} \alpha^3 + \alpha^2 (f-\alpha) + (f-\alpha)^3 \right)\\ 
&\quad + \alpha^2 \left( {\textstyle \frac{2\alpha}{3}}  - f \right) + \left( (f-\alpha)^3 + 3\alpha(f-\alpha)^2 \right)\\
&= 2 \Psi(f) + \alpha^2 \left( {\textstyle \frac{2 \alpha}{3}}  - f \right) + \left( (f-\alpha)^3 + 3\alpha(f-\alpha)^2 \right) \\
&>\Psi(f).
\end{align*}
In view of that \(\int \Psi(f_l)\dx = N_\Psi(f_l)\), the first part of the lemma and the lower bound on $\langle f_l, \Psi'(f_l)\rangle$ now follows from integrating the above expressions over \(\rone\).

For the upper bound, we have that
\begin{align*}
	\langle f_l, \Psi'(f_l)\rangle\leq  2 +\int_{f_l\geq \al} (f_l-\al)^3+3\al(f_l-\al)^2\dx \lesssim 2 +\int f_l^3\dx\lesssim 1,
\end{align*}
uniformly in $\al>0$ and $l$, where the last inequality follows from Lemma \ref{lemma:L2L3}.
 As shown above, we have
\begin{align*}
	\langle f_l, \Psi'(f_l) \rangle= & 2-\frac{1}{3}\int_{f_l<\al}f_l^3\dx \\
	&+\int_{f_l\geq \al}\al^2\left(\frac{2\al}{3}-f_l\right)+\left((f_l-\al)^3+3\al(f_l-\al)^2\right)\dx.
\end{align*}
Letting $B$ be as in Proposition~\ref{prop:upper bound on f_l}, we have that the integrand in the second integral is strictly negative for $\al\leq f_l<B$. Hence, if $f_l(0)\leq B$, then $\langle f_l, \Psi'(f_l) \rangle<2$. Assume therefore that $f_l(0)>B$. For any $\varepsilon>0$, we have by Proposition~\ref{prop:upper bound on f_l} that
\begin{equation*}
	\int_{f_l\geq B}\al^2\left(\frac{2\al}{3}-f_l\right)+\left((f_l-\al)^3+3\al(f_l-\al)^2\right)\dx<\varepsilon,
\end{equation*}
for all $l$ sufficiently large. Hence
\begin{align}
\langle f_l, \Psi'(f_l) \rangle< & 2-\frac{1}{3}\int_{f_l<\al}f_l^3\dx \label{eq:upper bound on f_l Psi(f_l)}\\
&+\int_{\al\leq f_l<B}\al^2\left(\frac{2\al}{3}-f_l\right)+\left((f_l-\al)^3+3\al(f_l-\al)^2\right)\dx+\varepsilon. \nonumber
\end{align}
In particular, $\langle f_l, \Psi'(f_l) \rangle<2+\varepsilon$ and using this estimate and the regularity of $K$, we have by Lemma \ref{lemma:local Euler} that
\begin{align*}
	|\Psi'(f_l(x))-\Psi'(f_l(y))|= & \frac{\langle f_l, \Psi'(f_l) \rangle}{J_l^2}|K\ast f_l(x)-K\ast f_l(y)| \\
	< & 3\al |K\ast f_l(x)-K\ast f_l(y)| \\
	\lesssim & \al \|f_l\|_\infty |x-y|^{1/2} \\
	\lesssim & \al^2 |x-y|^{1/2},
\end{align*}
uniformly in $\al$ and $l$, where we used that $K$ is smooth away from the origin and rapidly decaying, and $K(x)\simeq |x|^{-1/2}$ for $|x| \ll1$. As $\Psi'\in C^1$, this implies that if $f_l=\al$ at some point, then for any $0<c<\al$, the quantity $|\lbrace x : c\leq f_l(x)\leq \al \rbrace|$ is uniformly bounded below by a positive constant depending only on $c$ and $\al$. This implies that there is a $C>0$ independent of $l\gtrsim |\log(\al)|$ such that if $f_l(0)\geq \al$, then
\begin{equation*}
	\frac{1}{3}\int_{f_l<\al}f_l^3\dx>C.
\end{equation*}
As $\varepsilon>0$ above was arbitrary, we can choose $\varepsilon<C$ in \eqref{eq:upper bound on f_l Psi(f_l)} to conclude that $\langle f_l, \Psi'(f_l) \rangle<2$ for all $l$ sufficiently large.
\end{proof}

\begin{corollary}
	\label{cor: L1 and L3 bounds}
	For all $\al>0$ and all $l$ sufficiently large, the family \(\{f_{l}\}_l\) satisfies the estimate
	\begin{equation*}
	\|f_l\|_{L^1}+|xf_l(x)|\lesssim 1+\al^{-2},
	\end{equation*} 
	uniformly in $\al>0$ and $l>l_0$, where $l_0$ is as in Lemma \ref{lemma:inner product}.
\end{corollary}

\begin{proof}
By Lemma \ref{lemma:local Euler}, $\|f_l\|_{L^\infty}\lesssim \al$ uniformly in $l$ and the condition $N_\Psi (f_l)=1$ implies that $|\lbrace x : f_l(x)\geq \al \rbrace|\leq \frac{3}{2 \al^3}$ uniformly in $l$ (cf. the proof of Lemma \ref{lemma:inner product}). Hence
	\begin{equation*}
		\int_{f_l\geq \al} f_l \dx \lesssim \al^{-2},
	\end{equation*}
	uniformly in $l$. For $x\in (-2^l,2^l)$ such that $f_l(x)<\al$, we have by Lemma \ref{lemma:local Euler} that
	\begin{equation*}
		2\al f_l-f_l^2=\frac{\langle f_l, \Psi'(f_l)\rangle}{J_l^2} K\ast f_l.
	\end{equation*}
	We can rewrite this as
	\begin{equation*}
		\left[2 -\frac{\langle f_l, \Psi'(f_l)\rangle}{\al J_l^2} K\ast \right]f_l=\frac{f_l^2}{\al}.
	\end{equation*}
	By Lemma \ref{lemma:J-bound}, we have $J_l^2\geq\frac{1}{\al}\left(1+\frac{c_0}{1+\al^2}\right)$, and by Lemma \ref{lemma:inner product} $\langle \Psi'(f_{l}), f_{l} \rangle < 2$. Hence
	\begin{equation*}
	\Big|2 -\frac{\langle f_l, \Psi'(f_l)\rangle}{\al J_l^2}\Big| \gtrsim \frac{1}{1+\alpha^2},
	\end{equation*}
	uniformly in $\al>0$ and $l>l_0$. Hence
	\begin{equation*}
		\left\| \left[2 -\frac{\langle f_l, \Psi'(f_l)\rangle}{\al J_l^2} K\ast \right]^{-1}\right\|_{L^1\to L^1}\lesssim 1+\al^2,
	\end{equation*}
	and from the Euler--Lagrange equation we get
	\begin{align*}
		\int_{f_l<\al} f_l\dx \leq \left\| \left[2 -\frac{\langle f_l, \Psi'(f_l)\rangle}{\al J_l^2} K\ast \right]^{-1}\right\|_{L^1\to L^1} \frac{\|f_l\|_{L^2}^2}{\alpha}\lesssim 1+\alpha^{-2},
	\end{align*}
	uniformly in $\al>0$ and $l>l_0$, where we used that $\|f_l\|_{L^2}\lesssim \al^{-1/2}$ (cf. Lemma \ref{lemma:L2L3}). Adding the integrals over $f_l\geq \al$ and $f_l<\al$ together gives the uniform $L^1$ bound. As $f_l$ is bell-shaped, we have for $x\in (0,2^l)$,
	\begin{equation*}
		x f_l(x)\leq \int_0^x f_l(y)\dy \leq \|f_l\|_{L^1}\lesssim 1+\al^{-2}.
	\end{equation*}
	This concludes the proof.
\end{proof}

\subsection{The limit as \(l \to \infty\)}
We now proceed to find  the limit as $l\to \infty$.

\begin{lemma}\label{lemma:l-convergence}
For any $\al>0$, any sequence $\lbrace f_l\rbrace_l$ of local maximizers from Lemma \ref{lemma:local Euler} (recall that we have no proof of uniqueness) has a subsequence \(\{f_{l_k}\}_{k}\) that converges point-wise and in \(L^p(\rone)\), \(p > 1\), and the limit \(f_\alpha\) is a non-trivial bell-shaped solution of the global constrained maximization problem
\begin{equation}
\label{lo:50}
\sup\limits_{ N_{\Psi}(f)=1} \|K_{\f{1}{4}}*f\|_{L^2}.
\end{equation}  
The function \(f_\alpha\) furthermore satisfies the Euler--Lagrange equation
\begin{equation}
\label{lo:49}
\f{\dpr{f_\al}{\Psi'(f_\al)}} {J_\alpha^2}K*f_\alpha=\Psi'(f_\alpha).
\end{equation}
\end{lemma}

\begin{proof}
By the upper bound on $\|f_l\|_{L^\infty}$ in Lemma \ref{lemma:local Euler} and the estimates in Corollary \ref{cor: L1 and L3 bounds}, we have that for any fixed choice of $\al>0$, there exists an $l_0$ such that that
\begin{equation}
\label{eq:full sqrt estimate}
f_l(x)\lesssim \frac{1}{1+|x|} \qquad \text{ for all }\quad x,
\end{equation}
uniformly in $l>l_0$.
 
Next, since $\{f_l\}_l$ belongs to the unit sphere of $\cl^\Psi$, and \(\|f_l\|_{L^2(\rone)} \lesssim \alpha^{-1/2} N_\Psi(f_l)\) holds uniformly by Lemma~\ref{lemma:L2L3}, by weak compactness there is a weakly convergent subsequence such that, 
\[
f_{l_k} \rightharpoonup f_\alpha\in L^2(\rone). 
\]
By testing against characteristic functions, one sees that $f_\alpha$ is bell-shaped as well.

We want to obtain convergence in the Euler--Lagrange equation \eqref{eq:local Euler}. Since $\dpr{f_{l}}{\Psi'(f_{l})} \eqsim 1$ by Lemma~\ref{lemma:inner product}, and \(J_l\) is bounded from above and below by Lemmas~\ref{lemma:J_l limit} and~\ref{lemma:J-bound} for any fixed \(\alpha\), we may without loss of generality assume that $\lim_k  \dpr{f_{l_k}}{\Psi'(f_{l_k})}\neq 0$ and 
\[
J_\alpha = \lim_{k} J_{l_k}  \geq \frac{1}{\sqrt{\alpha}}
\] 
both exist.  The above convergence for \(f_{l_k}\) furthermore yields that
\begin{align*}
K*f_{l_k}(x) &=\int K(x-y) f_{l_k}(y) \dy \\
&=\dpr{K(x-\cdot)}{f_{l_k}}\to \dpr{K(x-\cdot)}{f_\alpha}=K*f_\alpha(x),
\end{align*}
for every fixed $x\in \rone$, since \(K \in L^q\) for \(q < 2\). Thus, we have established the point-wise convergence 
\[
\lim_{k \to \infty }K*f_{l_k}(x) = K*f_\alpha(x), \qquad x \in \rone.
\] 
Let $g_l =\Psi'(f_l)$. Given $x\in \rone$, for all $k$ sufficiently large such that $|x|<2^{l_k}$, we have that
\begin{equation}
\begin{aligned}
\label{lk:20}
g_{l_k}(x)&=\f{\dpr{f_{l_k}}{\Psi'(f_{l_k})}} {J_{l_k}^2}K*f_{l_k}(x)\\
&\to 
\f{\lim \dpr{f_{l_k}}{\Psi'(f_{l_k})}} {J_\alpha^2}K*f_\alpha(x) =: g_\alpha(x),
\end{aligned}
\end{equation}
converges point-wise as well. We can now readily deduce the point-wise convergence of $f_{l_k}$, since \(\Psi'\) is strictly positive away from the origin. Indeed, for \(|x| \leq 2^l\) the value of \(f_l(x) \) is given by $f_l(x) =  (\Psi')^{-1} \circ g_l(x)$.  More explicitly, 
\begin{equation}
\label{lk:85}
f_l(x)= \left(\alpha-\sqrt{\alpha^2-g_l(x)} \right)\chi_{g_l(x)\leq\alpha^2}+
  \Big(\alpha+\sqrt{\f{g_l(x)-\alpha^2}{3}}\Big)\chi_{g_l(x)>\alpha^2},
\end{equation}
whence the latter expression has a point-wise limit along the subsequence $l_k$.
By uniqueness, the point-wise and weak limits of \(f_{l_k}\) coincide, so that 
\[
\lim_{k \to \infty} f_{l_k}(x)=f_\alpha(x), 
\]
and \eqref{lk:85} thus holds with \(f_\alpha\) and \(g_\alpha\) exchanged for \(f_l\) and \(g_l\), respectively. Equivalently,  $g_\alpha=\Psi'(f_\alpha)$, so that \eqref{lk:20} implies that
\begin{equation}\label{lk:80}
\Psi'(f_\alpha)=g_\alpha=\f{\lim \dpr{f_{l_k}}{\Psi'(f_{l_k})}} {J_\alpha^2}K*f_\alpha.
\end{equation}
This equality is valid on the entire real line, since for each \(x \in \rone\), \eqref{lk:85} will eventually hold as \(l_k \to \infty\). 

It is now time to discuss the existence and value of  $\lim_{k\rightarrow \infty} \dpr{f_{l_k}}{\Psi'(f_{l_k})}$. 
By Lemma~\ref{lemma:inner product}, we have
\begin{align*}
\langle f_{l_k}, \Psi'(f_{l_k}) \rangle = &2 -  \frac{1}{3} \int_{f_{l_k} < \alpha} f_{l_k}^3 \dx\\ 
& + \int_{f_{l_k} \geq  \alpha} \alpha^2  \left(\frac{2\alpha}{3} - f_{l_k}\right)+(f_{l_k}-\alpha)^3 + 3\alpha(f_{l_k}-\alpha)^2 \dx,
\end{align*}
with \(f_{l_k}(x) \to f_\alpha(x)\) for all \(x\). By \eqref{eq:full sqrt estimate}, the sequence $\lbrace f_l\rbrace$ is dominated by a function that is in $L^p$ for all $p>1$. As $f_{l_k}$ converges point-wise, Lebesgue's dominated convergence theorem then implies that
\begin{equation*}
\lim_{k\rightarrow \infty} \int f_{l_k}^p\dx=\int f_\alpha^p\dx,
\end{equation*}
for all $p>1$. In particular, we get that
\begin{align}
	\lim_{k \to \infty} \langle f_{l_k}, \Psi'(f_{l_k}) \rangle = &2 -  \frac{1}{3} \int_{f_\al<\al} f_\alpha^3\dx, \nonumber\\
	&+ \int_{f_\al\geq \al} \al^2\left(\frac{2\al}{3}-f_\al\right)+(f_\al-\al)^3+3\al(f_\al-\al)^2\dx, \nonumber
\end{align}
and
\begin{equation}
\label{eq:lebesgue}
1=\lim_{k\rightarrow \infty} \int \Psi(f_{l_k})\dx=\int \Psi(f_\al)\dx.
\end{equation}
It follows that (cf. Lemma \ref{lemma:inner product}) $\lim_{k \to \infty} \langle f_{l_k}, \Psi'(f_{l_k}) \rangle= \langle f_{\al}, \Psi'(f_{\al}) \rangle$, and by \eqref{lk:80} we then have that
\begin{equation*}
\Psi'(f_\alpha)=\frac{\langle f_{\al}, \Psi'(f_{\al}) \rangle}{J_\alpha^2}K\ast f_\alpha.
\end{equation*}
Multiplying both sides by $f_\alpha$ and integrating, we get
\begin{equation*}
\langle f_{\al}, \Psi'(f_{\al}) \rangle=\frac{\langle f_{\al}, \Psi'(f_{\al}) \rangle}{J_\alpha^2} \langle K\ast f_\al,f_\al\rangle.
\end{equation*}
It follows that $J_\al^2= \dpr{K*f_\al}{f_\al}=\|K_{\f{1}{4}}*f_\al\|^2$. By \eqref{eq:lebesgue}, $N_\Psi(f_\alpha)=1$, and it follows that $f_\al$ is a maximizer.
\end{proof}

\begin{corollary}\label{cor:f_alpha}
For any $\al>0$, every bell-shaped maximizer $f_\al$ of the global constrained maximization problem satisfies
\begin{equation}
	\label{eq:upper-bound-f_alpha}
	f_\al(0)\leq \frac{4}{\sqrt{3}}\cos\left(\frac{5\pi}{18}\right)\alpha\approx 1.48 \alpha,
\end{equation}
\begin{equation*}
	1<\langle f_\al, \Psi'(f_\al)\rangle<2,
\end{equation*}
\begin{equation}
\label{lo:120}
\|f_\alpha\|_{L^1} + |x f_\alpha(x)| \lesssim 1+\al^{-2},
\end{equation}
\begin{equation}\label{eq:f_alpha nondegenerate}
\|f_\alpha\|_{L^3}^3 \gtrsim \frac{1}{1+\alpha^2},
\end{equation}
and for each fixed $\de>0$,
\begin{equation}
\label{lo:100}
\|f_\al\|_{L^\infty} \lesssim \al^{-\f{1}{2}+\delta}. 
\end{equation}
All these estimates are uniform in $\al>0$. Moreover, there is a number $\al_0\geq 0$ such that for all $\al>\al_0$, there are maximizers $f_\al$ satisfying $f_\al(0)<\al$.
\end{corollary}

\begin{proof}
The Euler-Lagrange equation \eqref{lo:49} and the smoothing effect of $K\ast$ implies that maximizers $f_\alpha$ are continuous, so if $f_\al(0)>\frac{4}{\sqrt{3}}\cos\left(\frac{5\pi}{18}\right)\alpha=B$, then
\begin{equation*}
	\int_{f_\alpha\geq B} |f_\alpha-B|\dx>0.
\end{equation*}
The upper bound \eqref{eq:upper-bound-f_alpha} then follows from Proposition~\ref{prop:upper bound on f_l} by taking $l\rightarrow \infty$. The bounds $1<\langle f_\al, \Psi'(f_\al)\rangle<2$, \eqref{lo:120} and \eqref{eq:f_alpha nondegenerate} then follow directly from Lemma~\ref{lemma:inner product}, Corollary~\ref{cor: L1 and L3 bounds} and Lemma~\ref{lemma:local Euler}, respectively, as the equivalent bounds for the local maximizers $f_l$ were uniform in $l$.

It only remains to show that the solutions $f_\al$ vanish in $L^\infty$ as $\al \to \infty$. For all $\al$ sufficiently large we have that $f_\al(0)<\al$, so by the Euler-Lagrange equation \eqref{lo:49} applied to $x=0$, and the bounds $J_\al\geq \al^{-1/2}$ and $\langle f_\al, \Psi'(f_\al) \rangle\eqsim 1$,  we obtain from Lemma~\ref{lemma:convolution_estimates} that
$$
2\al f_\al(0) - f_\al^2(0)=  \f{\langle f_\al, \Psi'(f_\al)\rangle} {J_\alpha^2} K*f_\al(0)\lesssim_q \al \|f_\al\|_{L^q}, 
$$
for all $q>2$. Interpolating between the $L^2$ and the $L^3$ bounds of Lemma~\ref{lemma:L2L3} we get that 
 $\|f_\al\|_{L^q}\lesssim \al^{1-\f{3}{q}}$, $q\in (2,3)$. Since this is available for all $q \in (2,3)$, we obtain the inequality 
 $$
 f_\al^2(0)-2\al f_\al(0)+C_\de \al^{\f{1}{2}+\de}\geq 0,
 $$
with \(0 < \delta \ll  1\) arbitrary small and \(C_\delta\) a positive constant depending on it. In view of that $f_\al(0) \leq \al$, the solution to the  above inequality is 
 $$
 f_\al(0) \leq \al-\sqrt{\al^2-  D_\de \al^{\f{1}{2}+\de}}\lesssim_\delta \al^{\delta-\f{1}{2}}. 
 $$
\end{proof}

\section{Dependence on the parameter \(\alpha\)}\label{sec:alpha}
In this section we investigate the dependence on the parameter \(\alpha\) for the maximizers $f_\al$. In particular we are interested in the maximizers that satisfy \(f_\alpha \leq \alpha\), as other maximizers are not solutions of the Whitham equation. We therefore introduce the threshold parameter
\begin{equation}
	\al_0:=\inf \{\xi>0 \colon \text{for any } \al \in (\xi, \infty) \text{ there exists a maximizer }   f_{\al}(0)<\al  \}.
\end{equation}
From Corollary~\ref{cor:f_alpha} it follows that $\al_0$ exists and is a finite number. Below we will prove that $\alpha_0>0$, meaning that there are maximizers found in this paper that either attain or exceed the height \(f_\alpha(0) = \alpha\). We will also prove that $\alpha_0<\frac{5}{2}$, which means that we have solutions also for intermediate (and perhaps small) values of $\alpha$. Note that \(\alpha\) is in fact in opposite relation to the wave-height of the final solution constructed in Section~\ref{sec:solitary-waves}; the solutions vanish as \(\alpha \to \infty\).

\begin{lemma}\label{lemma:lsc}
	$\alpha \mapsto \alpha J_\alpha^2$ is continuous and strictly decreasing on $(0,\infty)$, with
	\begin{align*}
	\lim_{\al\to \infty} \alpha J_\alpha^2 & =1, \\
	\lim_{\al\searrow 0} \alpha J_\alpha^2 & =\frac{3\tilde{B}^2}{2+3(\tilde{B}-1)+3(\tilde{B}-1)^3}>\frac{3}{2},
	\end{align*}
	where $\tilde{B}=\frac{4}{\sqrt{3}}\cos\left(\frac{5\pi}{18}\right)\approx 1.48$.
\end{lemma}

\begin{proof}
	Let $\al\in (0,\infty)$ and write $f_\al = \alpha q_\alpha( \alpha^3 \cdot)$. Then $N_\Psi(f_\al)=1$ implies that
	\begin{equation*}
	\int_{q_\al<1} q_\al^2 - \frac{1}{3} q_\al^3\dx +\int_{q_\al\geq 1}\frac{2}{3}+(q_\al-1)+(q_\al-1)^3\dx = 1,
	\end{equation*}
	with $0<q_\al\leq \al^{-1} \max(f_\al)$. From this we see that, in fact, for any $\tilde{\al}\in (0,\infty)$, we have $N_\Psi(\tilde{\al}q_\al(\tilde{\al}^3\cdot))=1$. Let $\al_1<\al$. As $\widehat{K}(\xi)$ is strictly decreasing in $|\xi|$, we have that $\widehat{K}(\al_1^3\xi)>\widehat{K}(\al^3\xi)$ for all $\xi\neq 0$, hence
	\begin{align*}
		\al_1 J_{\al_1}^2 &=  \al_1 \int \widehat{K}(\xi)|\widehat{f_{\al_1}}(\xi)|^2\dxi\\ 
		&\geq  \int \widehat{K}(\al_1^3 \xi)|\widehat{q_\al}(\xi)|^2\dx >\int \widehat{K}(\al^3 \xi)|\widehat{q_\al}(\xi)|^2\dxi=\al J_\al^2.
	\end{align*}
	As $\al\in (0,\infty)$ was arbitrary, this proves that $\al \mapsto \al J_\al^2$ is strictly decreasing. Similarly, we have that
	\begin{equation*}
	\al J_\al^2\geq \int \widehat{K}(\al^3 \xi)|\widehat{q_{\al_1}}(\xi)|^2\dxi,
	\end{equation*}
	and hence
	\begin{align*}
	0<\al_1 J_{\al_1}^2-\al J_\al^2\leq &\int \left(\widehat{K}(\al_1^3 \xi)-\widehat{K}(\al^3 \xi)\right)|\widehat{q_{\al_1}}(\xi)|^2\dxi \\
	\leq & |\alpha^3-\alpha_1^3| \int |\xi| \max \lbrace |\widehat{K}'(\alpha_1^3 \xi)|, |\widehat{K}'(\alpha^3 \xi)|\rbrace  |\widehat{q_{\al_1}}(\xi)|^2\dxi.
	\end{align*}
	As $|\xi \widehat{K}'(\xi)|$ is smooth and decaying and $\|q_\alpha\|_{L^2}^2\simeq 1$ uniformly in $\alpha$, this proves continuity (using regularity properties of $q_\alpha$ derived from the Euler-Lagrange equation, one can show that, in fact, the continuity is uniform).

	To see what happens as \(\alpha \to \infty\), we assume $f_\al\leq \al$, which holds for all $\al\geq\al_0$, and hence $0\leq q_\al\leq 1$; and rewrite the Euler--Lagrange equation \eqref{lo:49} with $f_\al=\al q_\al(\al^3 \cdot)$ in terms of \(q_\alpha\):
	\[
	2 q_\alpha(x) - q_\alpha^2(x) = \frac{2 - \frac{1}{3} \int q_\alpha^3 \dx}{\alpha J_\alpha^2} \frac{1}{\alpha^3} \int K\left( \frac{x-y}{\alpha^3}\right) q_\alpha(y)\dy.
	\]
	Since \(0 \leq q \leq 1\) and \(\alpha J_\alpha^2 \geq 1\), we see that for any \(p \in (1,2)\),
	\begin{equation}\label{eq:sup q}
	\|q_\alpha\|_{\infty} \leq 2 \|\alpha^{-3}K[ \alpha^{-3}(x-y)]\|_{L^p} \|q_\alpha\|_{L^{p^*}} \leq C \|K\|_{L^p}  \alpha^{-\frac{3(p-1)}{p}},
	\end{equation}
	where we have used that \(\int |q_\alpha(x)|^{p^*} \dx \leq \int q_\alpha^2(x) \dx \leq \frac{3}{2}\) when \(\frac{1}{p} + \frac{1}{p^{*}} = 1\) and \(p < 2\) (recall that \(K \not\in L^2\)). 
	By \eqref{eq:sup q},
	$$
	\| q_\alpha\|_{L^3}^3 \leq \|q_\al\|_{L^\infty} \| q_\alpha\|_{L^2}^2 \leq C_p \alpha^{-\frac{3(p-1)}{p}} \| q_\alpha\|_{L^2}^2 \to 0,
	$$
	and by the constraint \(N_\Psi(f_\alpha) = 1\) this in turn means that
	\[
	\lim_{\alpha \to \infty} \alpha \|f_\alpha\|_{L^2} = \lim_{\alpha \to \infty} \|q_\alpha\|_{L^2}^2 = 1.
	\]
	In general, as \(\widehat K \leq 1\), we have
	\begin{equation}\label{eq:upper alpha J^2 bound}
	\alpha J_\alpha^2 = \alpha \int \widehat{K}(\xi) |\widehat{f_\al}(\xi)|^2 \dxi = \int \widehat{K}(\al^3 \xi)|\widehat{q_\al}(\xi)|^2 \dxi\leq \|q_\al \|_{L^2}^2,
	\end{equation}
	and it follows that $\limsup_{\alpha \rightarrow \infty} \alpha J_\alpha^2\leq 1$. On the other hand, from Lemma \ref{lemma:J-bound} we know that $\alpha J_\alpha^2\geq 1$ for all $\alpha$, and we conclude that
	\[
	\lim_{\alpha \to \infty} \alpha J_\alpha^2= 1.
	\]

Now we turn to the limit $\alpha\searrow 0$. The condition $N_\Psi(f_\al)=1$ implies that (cf. \eqref{eq:derivative of h} and the subsequent discussion)
		\begin{equation*}
		\|q_\al\|_{L^2}^2\leq \frac{3\tilde{B}^2}{2+3(\tilde{B}-1)+3(\tilde{B}-1)^3},
		\end{equation*}
		with equality if and only if 
		\begin{equation*}
		q_\al =q=\tilde{B}\chi_{[-\frac{1}{2}(\frac{2}{3}+(\tilde{B}-1)+(\tilde{B}-1)^3)^{-1},\frac{1}{2}(\frac{2}{3}+(\tilde{B}-1)+(\tilde{B}-1)^3)^{-1}]}.
		\end{equation*}
		As $\widehat{K}$ is smooth and $\widehat{K}(0)=1$, we see that for any fixed choice of $q_\al=q$ for all $\al>0$, the last inequality in \eqref{eq:upper alpha J^2 bound} is achieved in the limit as $\al\searrow 0$. This proves the result.
\end{proof}
As an immediate corollary, we get that $\al_0>0$:
\begin{corollary}
	\label{cor:alpha_0>0}
	The threshold parameter $\al_0$ is strictly positive and $1<\al_0J_{\al_0}^2<\frac{3}{2}$.
\end{corollary}
\begin{proof}
	Assume that for every $\al>0$, there exists a maximizer $f_\al(0)<\al$ and let $f_\al = \alpha q_\alpha( \alpha^3 \cdot)$. Then $q_\al(0)<1$. In general, the condition $\int q^2-\frac{1}{3}q^3 \dx=1$, $q\leq 1$, implies that $\|q\|_{L^2}^2\leq \frac{3}{2}$ with equality if and only if $q=\chi_{[-\frac{3}{4},\frac{3}{4}]}$. Hence $\|q_\al\|_{L^2}^2<\frac{3}{2}$ and by \eqref{eq:upper alpha J^2 bound} we get
	\begin{equation*}
		\alpha J_\alpha^2<\frac{3}{2},
	\end{equation*}
	for all $\al>0$. This contradicts Lemma \ref{lemma:lsc}, and it follows that $\al_0>0$. 
\end{proof}

We can also provide a rough upper bound on $\alpha_0$. The estimates in the below proof may be improved, but the purpose of the proposition is to establish maximizers with $f_\al(0)<\al$ also for intermediate values of $\al$.

\begin{proposition}
	\label{prop:upper bound alpha_0}
	The threshold parameter satisfies $\al_0<\left(\frac{3}{2}\right)^{4/3}\left(\frac{2}{\pi}+1\right)^{2/3} \approx 2.385$.
\end{proposition}
\begin{proof}
	From Lemma \ref{lemma:J_l limit} we have $J_\al^2> \al^{-1}$ and by Young's inequality $\|K\ast f_\al\|_{L^\infty}\leq \|K\|_{L^{3/2}} \|f_\al\|_{L^3}$. Hence, if $f_\al\leq \al$, we get from the Euler-Lagrange equation that
	\begin{equation*}
		2\al f_\al -f_\al^2 =\frac{\langle f_\al, \Psi'(f_\al)\rangle}{J_\al^2}K\ast f_\al< \al \langle f_\al, \Psi'(f_\al)\rangle\|K\|_{L^{3/2}} \|f_\al\|_{L^3}.
	\end{equation*}
	If $f_\al\leq \al$, then the condition $N_\Psi(f_\al)$ implies that $\|f_\al\|_{L^3}^3<\frac{3}{2}$ and $\langle f_\al, \Psi'(f_\al)\rangle=2-\frac{1}{3}\|f_\al\|_{L^3}^3$, so $\langle f_\al, \Psi'(f_\al)\rangle \|f_\al\|_{L^3}<\left(\frac{3}{2}\right)^{4/3}$. So if $f_\al(0)=\al$, we get
	\begin{equation}
		\label{eq:upper bound alpha_0}
		\al^2< \left(\frac{3}{2}\right)^{4/3}\|K\|_{L^{3/2}}\al.
	\end{equation}
	As $K$ is positive, strictly monotone on $(0,\infty)$ and $\|K\|_{L^1}=1$, there is an $a>0$ such that $K(a)=1$ and
	\begin{align*}
		\|K\|_{L^{3/2}}^{3/2} & =2\int_0^\infty K(x)^{3/2} \dx \\
		& < 2\left(\int_0^{a} K(x)^{3/2}\dx+ \int_a^\infty K(x)\dx \right) \\
		&<2\int_0^{a} \left( K(x)^{3/2}-K(x) \right) \dx +1.
	\end{align*}
	Writing $\widehat{K}(\xi)=\frac{1}{\sqrt{|\xi|}}+\frac{\sqrt{\tanh(|\xi|)}-1}{\sqrt{|\xi|}}$, we have that the first term has inverse Fourier transform $1/\sqrt{2\pi |x|}$, while the second term is integrable and exponentially decaying and hence has a real-analytic transform. Moreover, the inverse transform of the second term is clearly negative around the origin. Hence $K(x)<\frac{1}{\sqrt{2\pi|x|}}$ for $|x|\leq a$ and $a<\frac{1}{2\pi}$, and we get that
	\begin{align*}
		\int_0^{a} K(x)^{3/2}-K(x)\dx &<\int_0^{a} (2\pi x)^{-3/4}-(2\pi x)^{-1/2}\dx\\ 
		&=\frac{4}{(2\pi)^{3/4}} a^{1/4}-\sqrt{\frac{2}{\pi}}a^{1/2}<\frac{1}{\pi}.
	\end{align*}
	It follows that $\|K\|_{L^{3/2}}<\left(\frac{2}{\pi}+1\right)^{2/3}$, and from \eqref{eq:upper bound alpha_0} we get that
	\begin{equation*}
		\al<\left(\frac{3}{2}\right)^{4/3}\left(\frac{2}{\pi}+1\right)^{2/3} \approx 2.385.
	\end{equation*}
	It follows that $\al_0$ must satisfy this upper bound.
\end{proof}

\begin{lemma}
	\label{lemma:precompactness}
	For $\al>0$, let $\lbrace \al_j\rbrace_j \subset (0,\infty)$ be a sequence that converges to $\al$ and let $f_{\al_j}$ be a maximizer corresponding to $\al_j$. Then a subsequence $\lbrace f_{\al_{j_k}}\rbrace_k$ converges point-wise and in $L^p$, simultaneously for all $p\in(1,\infty]$, to a maximizer of the corresponding maximization problem for $\al$.
\end{lemma}
\begin{proof}
	For ease of notation, let $f_{\al_j}=f_j$. By Lemma \ref{lemma:L2L3}, we have
	\begin{equation*}
		\|f_j\|_{L^2}\lesssim \frac{1}{\sqrt{\al_j}}\leq \frac{1}{\inf_j \sqrt{\al_j}}<\infty,
	\end{equation*}
	uniformly in $j$. Hence there is a subsequence $\lbrace f_{j_k}\rbrace_k$ and a $f_{\al}\in L^2$ such that
	\begin{equation*}
	f_{j_k}\rightharpoonup f_{\al}.
	\end{equation*}
	This implies that
	\begin{equation*}
	K\ast f_{j_k}(x)=\langle K(x-\cdot),f_{j_k}\rangle \rightarrow \langle K(x-\cdot),f_{\al}\rangle =K\ast f_{\al}(x)
	\end{equation*}
	converges point-wise. Let $\Psi_j$ be the $\Psi$-function corresponding to $\al_j$, and let $g_j=\Psi_j'(f_j)$. By Corollary \ref{cor:f_alpha}, $1<\langle f_{j_k},\Psi_{j_k}'(f_{j_k})\rangle<2$ for all $k$, so we can without loss of generality assume that $\lim_{k\rightarrow \infty} \langle f_{j_k},\Psi_{j_k}'(f_{j_k})\rangle>1$ exists. As $J_\alpha$ is continuous in $\alpha$, we get that
	\begin{align*}
	g_{j_k}(x)=\frac{\langle f_{j_k},\Psi_{j_k}'(f_{j_k})\rangle}{J_{\alpha_{j_k}}^2} K\ast f_{j_k}(x)\rightarrow \frac{\lim_k \langle f_{j_k},\Psi_{j_k}'(f_{j_k})\rangle}{J_{\alpha}^2} K\ast f_{\al}(x)=:g_{\al}(x)
	\end{align*}
	also converges point-wise. We have that
	\begin{equation*}
	f_j(x)= \left(\alpha_j-\sqrt{\alpha_j^2-g_j(x)} \right)\chi_{g_j(x)\leq\alpha_j^2}+
	\left(\alpha_j+\sqrt{\f{g_j(x)-\alpha_j^2}{3}}\right)\chi_{g_j(x)>\alpha_j^2}.
	\end{equation*}
	The right-hand side has a point-wise limit along the subsequence $j_k$, hence $f_{j_k}\rightarrow f_{\al}$ also converges point-wise. Moreover,
	\begin{equation*}
	g_{\al}=\Psi_{\al}'(f_{\al}).
	\end{equation*}
	By Corollary \ref{cor:f_alpha}, we have the following estimate, uniform in $j$:
	\begin{equation*}
		|x|f_j(x)\lesssim 1+\al_j^{-2}\leq 1+(\inf_j \al_j)^{-2}.
	\end{equation*}
	Hence we get that
	\begin{equation}
		\label{eq:bound on f(x)}
		f_j(x)\lesssim \frac{1}{1+|x|},
	\end{equation}
	uniformly in $j$. The right-hand side is in $L^p$ for all $p>1$, so the point-wise convergence and Lebesgue's dominated convergence theorem gives that
	\begin{equation*}
		\lim_{k\rightarrow \infty}\int f_j^p\dx =\int f_\al^p\dx<\infty,
	\end{equation*}
	for all $p>1$. This implies that $N_{\Psi_\al}(f_\al)=1$ and
	\begin{equation*}
		\lim_{k\to \infty} \langle f_{j_k},\Psi_{j_k}'(f_{j_k})\rangle= \langle f_\al, \Psi_\al'(f_\al)\rangle.
	\end{equation*}
	It follows that
	\begin{equation*}
		\Psi_\al'(f_\al)=\frac{\langle f_\al, \Psi_\al'(f_\al)\rangle}{J_\al^2}K\ast f_\al.
	\end{equation*}
	Multiplying by $f_\al$ on both sides and integrating over $\rn$, we get that $\|K_{\frac{1}{4}}\ast f_\al\|_{L^2}=J_\al$, and as $N_{\Psi_\al}(f_\al)=1$, this implies that $f_\al$ is a maximizer.
\end{proof}

\begin{corollary}
	\label{cor:maximizers for alpha_0}
	At the threshold parameter $\al_0$, there exists a maximizer $f_{\al_0}$ satisfying $f_{\al_0}(0)\leq \al_0$ and a maximizer $g_{\al_0}$ satisfying $g_{\al_0}(0)\geq \al_0$.
\end{corollary}

\begin{proof}
	Let $\lbrace \al_j\rbrace_j$ be a monotonically decreasing sequence with $\al_j\searrow \al_0$. By the definition of $\al_0$, we have that for all $j$ there are maximizers $f_j$ for $\al_j$ satisfying $f_j(0)<\al_j$. It follows from Lemma \ref{lemma:precompactness} that a subsequence of $\lbrace f_j\rbrace_j$ converges point-wise to a maximizer $f_{\al_0}$, which by the point-wise convergence must satisfy
	\begin{equation*}
		f_{\al_0}(0)\leq \al_0.
	\end{equation*}

	Similarly, letting $\lbrace \al_j\rbrace_j\subset (0,\al_0)$ be a monotonically increasing sequence with $\al_j\nearrow \al_0$, we have, for each $j$, maximizers $g_j$ satisfying $g_j(0)\geq \al_j$, and we conclude that there is a maximizer $g_{\al_0}$ such that
	\begin{equation*}
		g_{\al_0}(0)\geq \al_0.
	\end{equation*}
	This proves the result.
\end{proof}

\begin{remark}
	Note that if we have equality $f_{\al_0}(0)=\al_0$, then this will correspond to a solitary wave of maximal height for the Whitham equation (see Section \ref{sec:solitary-waves}). We would expect that $f_{\al_0}=g_{\al_0}$ in the above result, in which case $f_{\al_0}(0)=\al_0$. However, presently we have no uniqueness results or other results that precludes the possibility that the maximizers are distinct and $g_{\al_0}(0) >\al_0 > f_{\al_0}(0)$.
\end{remark}

\section{Solitary waves of the full-dispersion equation}
\label{sec:solitary-waves}
The following proposition shows that the range of bell-shaped solutions are confined to the `natural' region \((0,\frac{\mu}{2}]\) where the end-points correspond to the zero solution and the Whitham highest wave, respectively. All solitary waves have super-critical wave speed bounded from above by twice the critical wave speed, a result in line with both \cite{EGW11} and \cite{MR4002168}.

\begin{proposition}\label{prop:los} 
Any non-constant, bell-shaped solution of \eqref{eq:steady_whitham} satisfies \(0 < \varphi(x) \leq \frac{\mu}{2}\), with \(\sup \varphi \geq \mu - 1\) and
\begin{equation}\label{eq:09}
\varphi(x)=\f{\mu - \sqrt{\mu^2-4 K*\varphi(x)}}{2}
\end{equation}
almost everywhere.  If $\varphi \in L^1(\rone)$, then the wave speed \(\mu\) belongs to the open set~$(1,2)$.
 \end{proposition}   
 
\begin{proof} 
To prove \(0 < \varphi(x) \leq \frac{\mu}{2}\), denote
\[
  z(x)=K*\varphi(x).
\]
Then $z$ is bell-shaped, too, and in fact it is strictly decreasing in $(0, \infty)$ in accordance with the proof of Lemma~\ref{le:2.4}. From the quadratic equation $z=\mu \varphi - \varphi^2$ we conclude that 
\[
2\varphi(x)=
  \begin{cases}
\mu+\sqrt{\mu^2-4 z(x)}, \quad & x\in A_+, \\
\mu-\sqrt{\mu^2-4 z(x)}, \quad & x\in A_-,
  \end{cases}
\]
for some disjoint sets $A_\pm$ with $A_+\cup A_-=[0,\infty)$. Assume for a contradiction that $A_+$ has positive measure. Then there are sets $A^1_+$, $A^2_+$ in $A_+$ of positive measure  with $A^1_+$ to the left of $A^2_+$. Let $0<x_1<x_2$ be two arbitrary elements in $A^1$ and $A^2$, respectively.  Since $\varphi$ is bell-shaped, we have $\varphi(x_1)\geq \varphi(x_2)$. But solving the inequality 
\[
\mu+\sqrt{\mu^2-4 z(x_1)} \geq \mu+\sqrt{\mu^2-4 z(x_2)}
\]
yields $z(x_2)\geq z(x_1)$, so that $z|_{A^2_+}\geq z|_{A^1_+}$. This contradicts the strict monotonicity of \(z\), whence $A_+$ must have zero measure, and $\varphi=\varphi_-$ almost everywhere. Thus \eqref{eq:09} holds, $\varphi(x)\leq \f{\mu}{2}$, and \(\varphi\) is non-negative by assumption since it is bell-shaped, which gives that it is non-zero in view of that \(z = K \ast \varphi\) is strictly monotone. 

To prove the remaining bounds, note that for an integrable solution \(\varphi\) of \eqref{eq:steady_whitham},
  $$
  \int \varphi(x) \dx = \int K*\varphi(x) \dx=\mu \int \varphi(x) \dx  - \int \varphi(x)^2 \dx.
  $$
It follows that $(\mu-1) \int \varphi(x) \dx = \int \varphi(x)^2 \dx$, whence $\mu>1$. By taking the supremum in \eqref{eq:steady_whitham} one furthermore obtains 
\[
\mu \| \varphi\|_\infty \leq  \| K \ast \varphi\|_\infty +  \| \varphi^2 \|_\infty \leq \| \varphi\|_\infty +  \| \varphi \|_\infty^2. 
\]
This proves that \(\mu \leq 1 + \| \varphi \|_\infty\) for non-zero solutions. When \(\varphi \leq \frac{\mu}{2}\) we also obtain \(\mu \leq 2\) from the same inequality, with equality if and only if \(\|K \ast \varphi\|_\infty = \|\varphi\|_\infty\). For bell-shaped solutions, the latter means
\[
\int K(y) \varphi(y) \dy = \varphi(0),
\]  
which can happen only if \(\varphi\) equals \(\varphi(0)\) almost everywhere. Consequently, \(\mu < 2\) for non-constant solutions (this bound is general, and does not require integrability).
\end{proof}

\begin{theorem}\label{thm:main}
For any maximizer $f_\alpha$ satisfying $f_\alpha(0)\leq \alpha$, the function
\[
\varphi = \f{J_{\alpha}^2 f_{\alpha}}{2-\f{1}{3}\int f_{\alpha}^3 \dx}, \qquad 
\]
is a positive, even and  one-sided strictly monotone solution of the steady Whitham equation~\eqref{eq:steady_whitham}  with wave speed
\[
\mu =  \f{2\alpha J_{\alpha}^2}{2-\f{1}{3}\int f_{\alpha}^3(x) \dx} \in (1,2),
\]
satisfying \(0 < \varphi \leq \frac{\mu}{2}\). In particular, there is an injective curve of solutions parameterized by \(\alpha \in [\al_0,\infty)\), and the function 
\[
[\al_0,\infty) \ni \alpha \mapsto \alpha J_\alpha^2  \in (1, {\textstyle \frac{3}{2}})
\] 
is  strictly decreasing and continuous with the limit \(1\) achieved as $\al\rightarrow \infty$. For $\al>\al_0$ the waves satisfy $0<\varphi<\frac{\mu}{2}$ and these waves are all smooth, while for $\alpha = \alpha_0$ the wave satisfies $\varphi(0)\leq \frac{\mu}{2}$, with potential equality.
The solutions scale as
  \begin{equation}
  \label{eq:1050}
  \varphi \simeq \frac{f_\al}{\al},
  \end{equation}
  with the estimate being uniform in $\al$. In particular, the waves are small for large values of the parameter $\al$ in the sense that $\|\varphi\|_{L^p}\lesssim \al^{-1}$ for $\al\gg1$ and all $p\in [1,\infty]$. As \(\alpha \to \infty\), one further has \(\mu \to 1\), the bifurcation point for solitary waves.\\

\end{theorem}

\begin{remark}\label{rem:Lp-estimates}
The estimate for small waves could  be improved using an additional \(L^2\)-bound. That would yield the slightly better \(\|\varphi\|_{L^p} \lesssim \alpha^{-3/2}\) for \(p \geq 2\). 
\end{remark}

\begin{proof}
Let $\al\geq \al_0$. By the definition of $\al_0$, maximizers satisfy $f_\al(0)<\al$ for $\al>\al_0$ and by Corollary \ref{cor:maximizers for alpha_0}, there exists a maximizer $f_{\al_0}$ that satisfies $f_{\al_0}(0)\leq \al_0$. Then, by the Euler-Lagrange equation for the maximizers,
$$
K*f_{\al}= \f{J_{\al}^2}{2-\f{1}{3}\int f_{\al}^3 \dx} [2\al f_{\al} - f_{\al}^2],
$$
so by letting  $\beta=\f{2-\f{1}{3}\int f_{\al}^3(x) \dx}{J_{\al}^2}$, $\mu =  \frac{2\alpha}{\beta}$ and $\vp =\f{f_{\al}}{\be}$ one immediately sees that $\vp$ satisfies the Whitham equation with wave speed $\mu$. Since \(f_\alpha\) belongs to \(L^1\) (cf. Corollary \ref{cor:f_alpha}), is bell-shaped and one-sided strictly monotone, so is its scaling \(\varphi\), and we see that $\varphi$ satisfies $0<\varphi\leq \frac{\mu}{2}$, with equality being achieved if and only if $f_\al(0)=\al$, which can potentially only happen for $\al=\al_0$. Proposition~\ref{prop:los} then guarantees that \(\mu  \in (1,2)\). The smoothness of waves satisfying \(\varphi < \frac{\mu}{2}\) follows from \cite[Thm 5.1]{MR4002168}. The monotonicity, continuity and end-point limits of \(\alpha \to \alpha J_\alpha^2\) have been established in Lemma~\ref{lemma:lsc} and Corollary \ref{cor:alpha_0>0}. 

To find a full set of solutions, we start by showing that the map \(f_\alpha \to \varphi\) is injective. Let $\alpha_1, \alpha_2 \in [\al_0, \infty)$ (they can be equal) and say that \(f_{\alpha_1}\) and \(f_{\alpha_2}\),   would correspond to the same solution \(\varphi\). Then, by the definition of $\varphi$,
\begin{equation}\label{eq:injective}
\f{J_{\alpha_1}^2 f_{\alpha_1}}{2-\f{1}{3}\int f_{\alpha_1}^3 \dx} = \f{J_{\alpha_2}^2 f_{\alpha_2}}{2-\f{1}{3}\int f_{\alpha_2}^3 \dx}.
\end{equation}
That is, $f_{\alpha_2}$ is a scaling of $f_{\alpha_1}$; \(f_{\alpha_2} = \lambda f_{\alpha_1}\) for some $\lambda\in \rn_+$, and 
\[
J_{\alpha_2}^2 = (\mathcal{J}(f_{\alpha_2}))^2 = (\mathcal{J}(\lambda f_{\alpha_1}))^2 = \lambda^2 (\mathcal{J}(f_{\alpha_1}))^2 = \lambda^2 J_{\alpha_1}^2.
\]
Insertion of this into \eqref{eq:injective} yields a linear equation in \(\lambda^3\), which has as its only solution \(\lambda = 1\), since our maximizers are all positive. Thus the map \(f_\alpha \mapsto \varphi\) is injective. Next, note that for any $\alpha>0$ and any positive function $f\in L^2\cap L^3$ that satisfies $f\leq 2\alpha$ (which all maximizers do - cf. \eqref{eq:upper-bound-f_alpha}),
\begin{equation*}
\int \Psi_{\alpha}(f)\dx
\end{equation*}
is strictly increasing in $\alpha$. Hence the (possibly multi-valued) map \(\alpha \mapsto f_\alpha\) is also injective. Thus, if for each \(\alpha\in [\al_0,\infty)\) we pick one solution \(f_\alpha\), we find a full injective curve \([\al_0,\infty) \ni \alpha \mapsto \varphi\).

To get the estimates for $\varphi$, note that by the definition of $\varphi$ and $\mu$, we have 
\begin{equation*}
\varphi=\frac{1}{2\al} \mu f_\al,
\end{equation*}
and \eqref{eq:1050} then follows directly from the fact that $\mu\in (1,2)$. By Corollary \ref{cor:f_alpha}, $\|f_\al\|_{L^1}\lesssim 1$ for $\al>>1$ and $\|f_\al\|_{L^\infty}\lesssim \al^{-1/2+\delta}$ for each fixed $\delta>0$. Interpolation then gives the $L^p$-estimates for $\varphi$. As mentioned in Remark \ref{rem:Lp-estimates} this estimate can be improved for $p\geq 2$, for example by interpolation and the additional estimate $\|f_\al\|_{L^2}^2\lesssim \al^{-1}$. Finally, the limit \(\mu \to 1\) as \(\alpha \to \infty\) follows directly from the definition of \(\mu\) and Lemma~\ref{lemma:lsc}. 
\end{proof}
\begin{remark}
	We have the uniform bound $\|f_\al\|_{L^2}>J_\al>\al^{-1/2}$, and it follows that $\|\varphi\|_{L^2}>\frac{1}{2}\al^{-3/2}$. As we have solutions for all $\al\in [\al_0,\infty)$ and $\al_0<2.4$ (cf. Proposition \ref{prop:upper bound alpha_0}), it follows that the curve of solutions in Theorem \ref{thm:main} also includes solitary waves of intermediate size.
\end{remark}

\end{document}